\documentclass[12pt]{amsart}
\usepackage{amsfonts}
\usepackage{}
\usepackage{txfonts}
\usepackage{amssymb,amsmath,amsthm,mathrsfs}
\numberwithin{equation}{section}

\newtheorem{thm}{Theorem}[section]
\newtheorem{lemma}[thm]{Lemma}
\newtheorem{remark}[thm]{Remark}

\newtheorem{condition}[thm]{Condition}

\newcommand{\bt}{{\bf T}}
\newcommand{\bx}{{\bf X}}
\newcommand{\by}{{\bf y}}
\newcommand{\ba}{{\bf A}}
\newcommand{\bb}{{\bf B}}
\newcommand{\bd}{{\bf D}}
\newcommand{\bw}{{\bf W}}
\newcommand{\bi}{{\bf I}}

\newcommand{\re}{{\rm E}}
\newcommand{\rtr}{{\rm tr}}
\newcommand{\bbx}{{\bf x}}
\newcommand{\bbY}{{\bf Y}}
\newcommand{\cC}{{\mathcal C}}
\def\vec{\mbox{vec}}

\allowdisplaybreaks

\begin{document}

\title[CLT for LSS of general sample covariance matrix] {Central limit theorem for linear spectral statistics of large dimensional separable sample covariance matrices}

\author{ZHIDONG BAI,\ \ HUIQIN LI,\ \ Guangming Pan}
\thanks{ H. Q. Li was partially supported by China Scholarship Council;
}

\address{KLASMOE and School of Mathematics \& Statistics, Northeast Normal University, Changchun, P.R.C., 130024.}
\email{baizd@nenu.edu.cn}
\address{School of Mathematics and Statistics, Jiangsu Normal University, Xuzhou, P.R.C., 221116.}
\email{lihq118@nenu.edu.cn}
\address{Division of Mathematical Sciences, School of Physical and Mathmatical Sciences, Nanyang Technological University, Singapore 637371.}
\email{GMPAN@ntu.edu.sg}


\maketitle

\begin{abstract}
Suppose that $\mathbf X_n=(x_{jk})$ is $N\times n$ whose elements are independent real variables with mean zero, variance 1 and the fourth moment equal to three. The separable sample covariance matrix is defined as $\bb_n=\frac1N\bt_{2n}^{1/2}\bx_n\bt_{1n}\bx_n'\bt_{2n}^{1/2}$ where $\bt_{1n}$ is a symmetric matrix and $\bt_{2n}^{1/2}$ is a symmetric square root of the nonnegative definite symmetric matrix $\bt_{2n}$. 
Its linear spectral statistics (LSS) are shown to have Gaussian limits 
 when $n/N$ approaches a positive constant.


\end{abstract}

\section{Introduction}

The sample covariance matrix is one of the most commonly studied random matrices in Random Matrix Theory, which can be traced back to Wishart (1928) (see \cite{wishart1928generalised}). It plays an important role in multivariate analysis because many statistics in
traditional multivariate statistical analysis (e.g., principle component analysis, factor analysis and
multivariate regression analysis) can be written as functionals of the eigenvalues of sample covariance matrices.

Large dimensional data now appear in various fields such as finance and genetic experiments due to different reasons.
 To deal with such large-dimensional data, a new area in asymptotic statistics has been
developed where the data dimension p is no more fixed but tends to infinity together with
the sample size n. The random matrices proves to be a powerful tool for such large dimensional statistical problems. One may refer to the latest book in this area by J. F. Yao,
S. R. Zheng and Z. D. Bai (2015), the recent work by Ledoit and Wolf (2004) and Jiang and Yang (2013).

So far, most work focus on the sample covariance matrices of the form
  $${\bf S}_n=\frac1N{\bf T}_n^{1/2}\bx_n\bx_n'{\bf T}_n^{1/2}$$ where $\bx_n$ is a $N\times n$ matrix with independent entries and ${\bf T}_n$ is a nonnegative definite symmetric matrix. As we know ${\bf S}_n$ can be viewed as a sample covariance matrix formed from $n$ samples of the random vector  ${\bf T}_n^{1/2}\bbx_1$(where $\bbx_1$ denotes the first column of $\bx_n$, which has population
covariance matrix ${\bf T}_n$.
Much work has been done on the central limit theorem (CLT) for linear eigenvalues statistics of ${\bf S}_n$ under different assumptions. Among others we mention \cite{bai2004clt,J,JJ,Panzhou2008,Sch2011}. One of the key features of the above sample covariance matrices ${\bf S}_n$ is that the sample are independent. As far as we know there is no CLT available for the sample covariance matrices generated from the dependent sample.

In view of the above we consider a kind of general sample covariance matrices
\begin{equation}\label{a1}
\bb_n=\frac1N\bt_{2n}^{1/2}\bx_n\bt_{1n}\bx_n'\bt_{2n}^{1/2},
\end{equation}
 where $\bt_{2n}$ is $N\times N$ nonnegative definite symmetric matrix and $\bt_{1n}$ is $n\times n$ symmetric. This model finds applications in the diverse fields including spatio-temporal statistics, wireless communications and
 econometrics. For example, the data matrix can be represented as
\begin{equation}\label{a2}
\bbY_n=\bt_{2n}^{1/2}\bx_n\bt_{1n}^{1/2}
\end{equation}
if $\bt_{1n}$ is nonnegative definite symmetric. Denote by $\vec(\bbY_{n})$ the vector operator that stacks the columns of $\bbY_{n}$ into
a column vector. This model is referred to as a separable covariance model because
the covariance of $\vec(\bbY_{n})$ is the Kronecker product of
$\bt_{1n}$ and $\bt_{2n}$. The rows of the data matrix $\bbY_n$ correspond to indices of spatial locations and the column
indices correspond to points in time in the field of spatio-temporal statistics. This covariance structure implies that the entries of $\bbY_n$ are correlated in time (column),
but the pattern of temporal correlation does not change with location (row). One may see \cite{PS} and the references therein. 

In econometrics, when determining the number of factors in the approximate factor models \cite{AO} assumes that the idiosyncratic components of the data is of the form $\bbY_n$.  This allows the idiosyncratic terms to be non-trivially correlated both cross-sectionally and over time. The cross-sectional correlation is caused by matrix $\bt_{2n}^{1/2}$ linearly
combining different rows of $\bx_n$, whereas the correlation over time is caused by matrix
$\bt_{2n}^{1/2}$ linearly combining different columns of $\bx_n$.

Another motivation of considering the sample covariance matrices $\bb_n$ is the matrix data.  Matrix observations are becoming increasingly available due to the rapid advance in the information
technology. For example, images are routinely stored as pixel by pixel data;
agricultural exports can be represented via matrices, one for each year, with rows
denoting for example different regions and columns different produces; the gene expression of a single subject can be
organized as a matrix with the rows for tissue types and the columns for genes.
There is an abundance
of data that can be characterized as matrix variates in food sciences and chemometrics. In general, the sample covariance matrix of the matrix data
is
$$
\frac{1}{nm}\sum\limits_{k=1}^m\bbY_{nk}\bbY_{nk}',
$$
where $\bbY_{nk},k=1,\cdots,m$ are $n\times N$ matrix data. 
Some papers argued that for many
matrix variates, it is more appropriate to assume that
$$Cov(\vec(\bbY_{nk}))=\bt_{1n}\otimes \bt_{2n}.
$$
One may refer to \cite{Leng:Tang:2012} and the references therein. Here, the matrix data $\bbY_n$ defined in (\ref{a2}) is just one matrix observation for simplicity. Moreover, write
\begin{align*}
\frac{1}{nm}\sum\limits_{k=1}^m\bbY_{nk}\bbY_{nk}'=\frac{1}{nm}\bt_{2n}^{1/2}\left(\bx_{n1},\cdots,\bx_{nm}\right)\left(\bi_m\otimes\bt_{1n}\right)
\left(\bx_{n1},\cdots,\bx_{nm}\right)'\bt_{2n}^{1/2}.
\end{align*}
Let $\widetilde\bt_{1n}=\bi_m\otimes\bt_{1n}$ and $\widetilde\bx_n=\left(\bx_{n1},\cdots,\bx_{nm}\right)$. Then
\begin{align*}
\frac{1}{nm}\sum\limits_{k=1}^m\bbY_{nk}\bbY_{nk}'=\frac{N}{nm}\frac1N\bt_{2n}^{1/2}\widetilde\bx_n\widetilde\bt_{1n}
\widetilde\bx_n'\bt_{2n}^{1/2}
\end{align*}
which has the same form as $(\ref{a1})$ if $nm$ and $N$ are of the same order.



For any hermitian matrix ${\bf A}$ of size $n\times n$ its empirical spectral distribution (ESD) is defined by
\begin{align*}
F^{{\bf A}}(x)=\frac1n\sum_{j=1}^nI(\lambda_j\le x),
\end{align*}
where $\{\lambda_j\}$ are eigenvalues of $\bf A$.
For $\bb_n$ defined in (\ref{a1}), a number of papers (\cite{MKV} and \cite{zlx}) investigated its empirical spectral
distribution $F_{\bb_n}$ and the weakest assumption is given in \cite{zlx}, which is specified below. To characterize its limit
define the Stieltjes transform of any distribution function $F^{{\bf A}}(x)$ to be
$$m_{F^{{\bf A}}}(z)=\int\frac1{x-z}dF^{{\bf A}}(x)=\frac1n\rtr({\bf A}-z\bi)^{-1},\quad z\in\mathbb{C}^+.$$
Throughout the paper we make the following assumption.
\begin{condition}\label{def}
\begin{itemize}
\noindent\item[(i)] $\bx_n=(x_{jl})$ is $N\times n$ consisting of independent real random variables with $\re x_{jl}=0,\re x_{jl}^2=1$, 
            satisfying for each $\delta>0$, as $n\to\infty$
             \begin{align*}
    \frac1{\delta^2nN}\sum_{j,l}\re\left(x_{jl}^2I\left(|x_{jl}|>\delta\sqrt n\right)\right)\to0.
    \end{align*}
\item[(ii)] $\bt_{1n}$ is $n\times n$ real symmetric matrix (without loss of generality, we assume that $\bt_{1n}$ is not semi-negative definite) and $\bt_{2n}$ is $N\times N$ nonnegative definite real symmetric matrix. 
\item[(iii)] With probability 1, as $n\to\infty$, the empirical spectral distributions of $\bt_{1n}$ and $\bt_{2n}$, denoted by $H_{1n}$ and $H_{2n}$ respectively, converge weakly to two probability functions $H_1$ and $H_2$, respectively.
\item[(iv)] $N=N(n)$ with $n/N\to c>0$.
\item[(v)] $\bx_n,\bt_{1n},\bt_{2n}$ are independent.
\end{itemize}
\end{condition}

 L. X. Zhang \cite{zlx} establishes the following conclusion under Condition \ref{def}. 
 For $\bb_n$ defined in (\ref{a1}), with probability $1$, as $n\to\infty$,  the ESD of $\bb_n$ converges weakly to a non-random probability distribution function $F$ for which if $H_1=1_{[0,\infty)}$ or $H_2=1_{[0,\infty)}$, then $F=1_{[0,\infty)}$; otherwise the Stieltjes transform $m(z)$ of $F$ is determined by the following system of equations (\ref{eqs}), where for each $z\in\mathbb{C}^+$,
\begin{align}\label{eqs}
\begin{cases}
s(z)=-z^{-1}(1-c)-z^{-1}c\int\frac1{1+q(z)x}dH_1(x)\\
s(z)=-z^{-1}\int\frac1{1+p(z)y}dH_2(y)\\
s(z)=-z^{-1}-p(z)q(z).
\end{cases}
\end{align}
Then, the Stieltjes transfrom $m(z)$ of $F$, together with the two other functions, denoted by $g_1(z)$ and $g_2(z)$, $\left(m(z),g_1(z),g_2(z)\right)$ is the unique
solution to (\ref{eqs}) in the set
\begin{align*}
U=\bigg\{\left(s(z),p(z),q(z)\right):\Im s(z)>0,\Im (zp(z))>0,\Im q(z)>0\bigg\}
\end{align*}
where $\Im h(z)$ stands for the imaginary part of $h(z)$. Denote $\underline{\bf B}_n=\frac1N\bt_{1n}{\bx}_{n}'{\bf T}_{2n}{\bx}_{n}$. Then we have the following relationship between the empirical distributions of $\bb_n$ and $\underline{\bf B}_n$
\begin{align*}
F^{\bb_n}(x)=c_nF^{\underline{\bf B}_n}(x)+(1-c_n)I_{[0,\infty)}(x),
\end{align*}
and hence
\begin{align}\label{al1}
m_n(z)=c_n\underline m_n(z)+z^{-1}(c_n-1).
\end{align}
where $c_n=n/N$, $m_n(z)=m_{F^{\bb_n}}(z)$ and $\underline m_n(z)=m_{F^{\underline{\bf B}_n}}(z)$. Denote by $\underline F$ the limiting distribution of $F^{\underline{\bf B}_n}$. Then $F$ and $\underline F$ must satisfy
\begin{align*}
F(x)=c\underline F(x)+(1-c)I_{[0,\infty)}(x),
\end{align*}
and
\begin{align}\label{al3}
m(z)=c\underline m(z)-z^{-1}(1-c)
\end{align}
where $\underline m(z)=m_{\underline F}(z)$. If we let $F^{c,H_{1},H_{2}}$ denote $F$, then $F^{c_n,H_{1n},H_{2n}}$ is obtained from $F^{c,H_{1},H_{2}}$ with ${c,H_{1},H_{2}}$ replaced by $c_n,H_{1n},H_{2n}$ respectively. Let $m_n^0(z)=m_{F^{c_n,H_{1n},H_{2n}}}(z)$ for simplicity. Moreover $g_{1n}^0(z)$ and $g_{2n}^0(z)$ are similarly obtained from $g_1(z)$ and $g_2(z)$ respectively. Then $\left(m_n^0(z),g_{1n}^0(z),g_{2n}^0(z)\right)$ satisfies the equations (\ref{eqs}). In other words
\begin{align}
\underline m_n^0(z)=&-z^{-1}\int\frac1{1+g_{2n}^0(z)x}dH_{1n}(x)\label{i1}\\
m_n^0(z)=&-z^{-1}\int\frac1{1+g_{1n}^0(z)y}dH_{2n}(y)\label{i2}\\
m_n^0(z)=&-z^{-1}-g_{1n}^0(z)g_{2n}^0(z).\label{i3}
\end{align}
Furthermore,
\begin{align}
&zg_{1n}^0(z)=-c_n\int\frac x{1+g_{2n}^0(z)x}dH_{1n}(x)\label{i4}\\
&zg_{2n}^0(z)=-\int\frac y{1+g_{1n}^0(z)y}dH_{2n}(y)\label{i5}.
\end{align}

\cite{CH} further investigated the limiting spectral measure of $\bb_n$ and \cite{PS} proved that no eigenvalues exist outside the support of limiting empirical spectral
distribution of $\bb_n$. But \cite{PS} required $\bt_{2n}$ in $\bb_n$ to be diagonal (with positive diagonal entries). It is well known that many important statistics in multivariate analysis can be written as functionals of the ESD of some random matrices. 
 In view of this the aim of this paper is to establish the central limit theorem for linear spectral statistics (LSS) of $\bb_n$. LSS of general sample covariance matrices are quantities of the form
\begin{align*}
\frac1N\sum_{j=1}^Nf(\lambda_j^{\bb_n})=\int f(x)dF^{\bb_n}(x)
\end{align*}
where $f$ is some continuous and bounded real function on $(-\infty,\infty)$.

This paper is organized as follows. Section 2 establishes the main result about the CLT for LSS of $\bb_n$. By the Stieltjes transform method, we complete the proof of theorem when the entries of matrix are Gaussian variables in Section 3. Section 4 extends the result from the Gaussian case to the general case through comparing their characteristic functions. Some useful lemmas are listed in Appendix.

\section{Main result}

Define
$$
G_n(x)=N\Big(F^{\bb_n}(x)-F^{c_n,H_{1n},H_{2n}}\Big).
$$
The main result is stated in the following theorem.
\begin{thm}\label{th1}
Denote by $s_1\ge\cdots\ge s_n$ ($s_1> 0$) the eigenvalues of $\bt_{1n}$. Let $f_1,\cdots,f_{\kappa}$
be functions on $\mathbb{R}$ analytic on an open interval containing
\begin{align}\label{int}
\Bigg[&\liminf_ns_n\left(\lambda_{\min}^{\bt_{2n}}I_{(0,1)}(c)\left(1-\sqrt c\right)^2I(s_n\ge0)+\lambda_{\max}^{\bt_{2n}}\left(1+\sqrt c\right)^2I(s_n<0)\right),\notag\\
&\limsup_ns_1\left(\lambda_{\max}^{\bt_{2n}}\left(1+\sqrt c\right)^2\right)\Bigg].
\end{align}
In addition to Condition \ref{def} we further suppose that $\re x_{jl}^4=3$ and
            for each $\delta>0$, as $n\to\infty$
    \begin{align*}
    \frac1{\delta^4n^2}\sum_{j,l}\re\left(x_{jl}^4I\left(|x_{jl}|>\delta\sqrt n\right)\right)\to0.
    \end{align*}
    Also suppose that the spectral norms of $\bt_{1n}$ and $\bt_{2n}$ are both bounded in n.
Then
$$\left(\int f_1(x)dG_n(x),\cdots,\int f_{\kappa}(x)dG_n(x)\right)$$
converges weakly to a Gaussian vector $\left(X_{f_1},\cdots,X_{f_{\kappa}}\right)$ with mean
\begin{align}\label{ee}
\re X_{f}=&-\frac1{2\pi i}\oint_{\cC} f(z)\Bigg\{\left(d_1(z)-d_2(z)\right)\Bigg\{1-z^{-1}\left[\int\frac{x}{\left({1+xg_{2}(z)}\right)^2}dH_{1}(x)\right]^{-1}\\
&\times\int\frac{x^2}{\left({1+xg_{2}(z)}\right)^2}dH_{1}(x)\int\frac{t}
{\left({1+g_{1}(z)t}\right)^2}dH_{2}(t)\Bigg\}^{-1}\Bigg\}dz\notag
\end{align}
and covariance function
\begin{align}\label{cc}
{\rm Cov}\left( X_{f},X_g\right)=-\frac1{2\pi^2}\oint_{\cC_1}\oint_{\cC_2}\frac{\partial^2}{\partial z_2\partial z_1}\int_0^{f(z_1,z_2)}\frac{1}{1-z}dzdz_1dz_2
\end{align}
where $f,g\in\left\{f_1,\cdots,f_{\kappa}\right\}$. Here
\begin{align*}
 d_1(z)=&-cz^{-3}\int\frac{x^2}{\left({1+xg_{2}(z)}\right)^2}dH_{1}(x)\int\frac{t^2}{\left(g_{1}(z)t+1\right)^3}dH_{2}(t)\\
&\times\left[1-cz^{-2}\int\frac{x^2}{\left({1+xg_{2}(z)}\right)^2}dH_{1}(x)\int\frac{t^2}
{\left(g_{1}(z)t+1\right)^2}dH_{2}(t)\right]^{-1}\\
&-cz^{-4}\int\frac{x^3}{\left({1+xg_{2}(z)}\right)^3}dH_{1}(x)\int\frac{t}{\left(g_{1}(z)t+1\right)^2}dH_{2}(t)
\int\frac{t^2}{\left(g_{1}(z)t+1\right)^2}dH_{2}(t)\\
&\times\left[1-cz^{-2}\int\frac{x^2}{\left({1+xg_{2}(z)}\right)^2}dH_{1}(x)\int\frac{t^2}
{\left(g_{1}(z)t+1\right)^2}dH_{2}(t)\right]^{-1},
\end{align*}
\begin{align*}
d_2(z)=&{- c}z^{-4}\int\frac{x^2}{\left(1+xg_{2}(z)\right)^3}dH_{1}(x)\int\frac{t^2}{\left(g_{1}(z)t+1\right)^2}dH_{2}(t)\\
&\times\left[\int\frac{x}{\left({1+xg_{2}(z)}\right)^2}dH_{1}(x)\right]^{-1}
\int\frac{x^2}{\left({1+xg_{2}(z)}\right)^2}dH_{1}(x)\int\frac{t}
{\left({1+g_{1}(z)t}\right)^2}dH_{2}(t)\\
&\times\left[1-cz^{-2}\int\frac{x^2}{\left({1+xg_{2}(z)}\right)^2}dH_{1}(x)\int\frac{t^2}
{\left(g_{1}(z)t+1\right)^2}dH_{2}(t)\right]^{-1},
\end{align*}
and
\begin{align*}
f(z_1,z_2)=\frac1{z_1z_2}\frac{z_1g_1(z_1)-z_2g_1(z_2)}
{g_2(z_1)-g_2(z_2)}\frac{z_1g_2(z_1)-z_2g_2(z_2)}{g_1(z_1)-g_1(z_2)}.
\end{align*} The contours in (\ref{ee}) and (\ref{cc}) (two contours in (\ref{cc}), which we may assume to be nonoverlapping) are closed and are taken in the positive direction in the complex plane, each enclosing the support of $F^{c,H_1,H_2}$.
\end{thm}
\begin{remark}
It is worth mentioning that our result is consistent with that in \cite{bai2004clt}. We distinguish two cases to show the consistency according to whether $\bt_{2n}$ or $\bt_{1n}$ reduces to the identity matrix.

When $\bt_{2n}=\bi$ and $\bt_{1n}$ is a nonnegative definite symmetric matrix, $\bb_n=\frac1N\bx_n\bt_{1n}\bx_n'$. Then (\ref{eqs}) is transformed into
\begin{align*}
\begin{cases}
\underline m(z)=-z^{-1}\int\frac1{1+m(z)x}dH_1(x)\\
g_1(z)=-\frac1{zm(z)}-1\\
g_2(z)=m(z).
\end{cases}
\end{align*}
It follows that
\begin{align*}
\re X_{f}=&-\frac1{2\pi i}\oint_{\cC} f(z)\int\frac{cm(z)^3x^2}{\left(1+xm(z)\right)^3}dH_1(x)\Bigg\{1-\int\frac{cm(z)^2x^2}{\left(1+xm(z)\right)^2}dH_1(x)\Bigg\}^{-2}dz
\end{align*}
and
\begin{align*}
f(z_1,z_2)=1+\frac{m(z_1)m(z_2)\left(z_1-z_2\right)}{m(z_2)-m(z_1)}.
\end{align*}
These are the same as those in \cite{bai2004clt}.

If $\bt_{1n}=\bi$ then $\bb_n=\frac1N\bt_{2n}^{1/2}\bx_n\bx_n'\bt_{2n}^{1/2}$. Let $\widetilde\bb_n=\frac1n\bt_{2n}^{1/2}\bx_n\bx_n'\bt_{2n}^{1/2}$ and $\underline{\widetilde\bb_n}=\frac1n\bx_n'\bt_{2n}\bx_n$. We use ${\widetilde m}_n(z)$ and $\underline{\widetilde m}_n(z)$ to denote the Stieltjes transforms of ${F^{{\widetilde\bb_n}}}$ and ${F^{\underline{\widetilde\bb_n}}}$ respectively. Denote by $\widetilde F^{c^{-1},H_2}$ the limiting distribution of $\widetilde F^{\widetilde\bb_n}$. Moreover $\widetilde F^{c_n^{-1},H_{2n}}$ is obtained from $\widetilde F^{c^{-1},H_2}$ with $c,H_2$ replaced by $c_n,H_{2n}$ respectively. Let ${\widetilde m}(z)=\lim_{n\to\infty}{\widetilde m}_n(z),\underline{\widetilde m}(z)=\lim_{n\to\infty}\underline{\widetilde m}_n(z)$ and $\underline{\widetilde m}_n^0(z)=m_{\widetilde F^{c_n^{-1},H_{2n}}}(z)$. Due to (\ref{al:4}) below we only need to consider the limiting distribution of $\widetilde M_n(z)=N[{\widetilde m}_n(z)-{\widetilde m}_n^0(z)]$. Firstly, (\ref{eqs}) becomes
\begin{align*}
\begin{cases}
 m(z)=-z^{-1}\int\frac1{1+c\underline m(z)x}dH_2(x)\\
g_1(z)=c\underline m(z)\\
g_2(z)=-\frac1{z\underline m(z)}-1
\end{cases}.
\end{align*}
By Lemma \ref{th2} below and the above equations, we have
\begin{align}\label{eq3}
\re M(z)=&\int\frac{c\underline m(z)^3x^2}{\left(1+cx\underline m(z)\right)^3}dH_2(x)\Bigg\{1-\int\frac{c\underline m(z)^2x^2}{\left(1+cx\underline m(z)\right)^2}dH_2(x)\Bigg\}^{-2}
\end{align}
and
\begin{align*}
f(z_1,z_2)=1+\frac{\underline m(z_1)\underline m(z_2)\left(z_1-z_2\right)}{\underline m(z_2)-\underline m(z_1)}.
\end{align*}
Note that $\bb_n=c_n\widetilde \bb_n$. It can be verified that
\begin{align*}
\underline {\widetilde m}_n(z/c_n)=c_n \underline m_n(z)
\end{align*}
and
\begin{align*}
 M_n(z)=c_n^{-1}\widetilde M_n(z/c_n).
\end{align*}
These imply that
\begin{align}\label{eq1}
\underline {\widetilde m}(z/c)=c \underline m(z)
\end{align}
and
\begin{align}\label{eq2}
 M(z)=c^{-1}\widetilde M(z/c)
\end{align}
where $\widetilde M(z)$ is a two-dimensional Gaussian process, the limit of weak convergence of $\widetilde M_n(z)$. Plugging (\ref{eq1}) and (\ref{eq2}) into (\ref{eq3}), one has
\begin{align*}
\re \widetilde M(z/c)=&c^{-1}\int\frac{\underline {\widetilde m}(z/c)^3x^2}{\left(1+x\underline {\widetilde m}(z/c)\right)^3}dH_2(x)\Bigg\{1-c^{-1}\int\frac{\underline {\widetilde m}(z/c)^2x^2}{\left(1+x\underline {\widetilde m}(z/c)\right)^2}dH_2(x)\Bigg\}^{-2}
\end{align*}
and
\begin{align*}
f(z_1,z_2)=1+\frac{\underline {\widetilde m}(z_1/c)\underline {\widetilde m}(z_2/c)\left(z_1/c-z_2/c\right)}{\underline {\widetilde m}(z_2/c)-\underline {\widetilde m}(z_1/c)}.
\end{align*}
Hence the expectation and covariance are the same as those in Bai and Siverstein (2004).
\end{remark}

By Cauchy's formula
\begin{align}\label{al:4}
\int f(x)dG(x)=-\frac1{2\pi i}\oint f(z)m_{G}(z)dz
\end{align}
where $G$ is a cumulative distribution function (c.d.f.) and $f$ is analytic on an open set containing the support of $G$. The complex integral on the right-hand side is over any positively oriented contour enclosing the support of $G$ and on which $f$ is analytic. Hence, the proof of Theorem \ref{th1} relies on establishing limiting results on
\begin{align*}
M_n(z)=N\left[m_n(z)-m_n^0(z)\right].
\end{align*}
The contour $\mathcal{C}$ is defined as follows.

By Condition \ref{def}, we may suppose $\max\left\{\left\|\bt_{1n}\right\|,\left\|\bt_{2n}\right\|\right\}\le\tau$.
Let $v_0$ be any positive number. Let $x_r$ be any positive number if the right end point of interval (\ref{int}) is zero. Otherwise choose
\begin{align*}
x_r\in
(\limsup_ns_1\lambda_{\max}^{\bt_{2n}}\left(1+\sqrt c\right)^2,\infty).
\end{align*}
Let $x_l$ be any negative number if the left end point of interval (\ref{int}) is zero. Otherwise choose
\begin{align*}
x_l\in
\begin{cases}
(0,\liminf_ns_n\lambda_{\min}^{\bt_{2n}}I_{(0,1)}(c)\left(1-\sqrt c\right)^2),&{\rm if} \ \liminf_ns_n\lambda_{\min}^{\bt_{2n}}I_{(0,1)}(c)>0,\\
(-\infty,\liminf_ns_n\lambda_{\max}^{\bt_{2n}}\left(1+\sqrt c\right)^2),&{\rm if} \ \liminf_ns_n\lambda_{\min}^{\bt_{2n}}I_{(0,1)}(c)\le 0.
\end{cases}
\end{align*}
Let
\begin{align*}
\mathcal{C}_u=\left\{x+iv_0:x\in[x_l,x_r]\right\}.
\end{align*}
Define the contour $\mathcal{C}$
\begin{align*}
\mathcal{C}=\left\{x_l+iv:v\in[0,v_0]\right\}\cup\mathcal{C}_u\cup\left\{x_r+iv:v\in[0,v_0]\right\}.
\end{align*}

To avoid dealing with the small $\Im z$, we truncate $M_n(z)$ on a contour $\mathcal{C}$ of the complex plane. We define now the subsets $\mathcal{C}_n$ of $\mathcal{C}$ on which $M_n(\cdot)$ agrees with $\widehat M_n(\cdot)$. Choose sequence $\{\varepsilon_n\}$ decreasing to zero satisfying for some $\alpha\in(0,1)$
\begin{align*}
\varepsilon_n\ge n^{-\alpha}.
\end{align*}
Let
\begin{align*}
\mathcal{C}_l=\left\{x_l+iv:v\in[n^{-1}\varepsilon_n,v_0]\right\}\quad{\rm and}\quad
\mathcal{C}_r=\left\{x_r+iv:v\in[n^{-1}\varepsilon_n,v_0]\right\}.
\end{align*}
Then $\mathcal{C}_n=\mathcal{C}_l\cup\mathcal{C}_u\cup\mathcal{C}_r$. For $z=x+iv$, the process $\widehat M_n(\cdot)$ can now be defined as
\begin{align}\label{aa}
\widehat M_n(\cdot)=
\begin{cases}
M_n(z),&{\rm for} \ z\in\mathcal{C}_n,\\
M_n(x_l+in^{-1}\varepsilon_n),& {\rm for} \ x=x_l,v\in[0,n^{-1}\varepsilon_n],\\
M_n(x_r+in^{-1}\varepsilon_n),& {\rm for} \ x=x_r,v\in[0,n^{-1}\varepsilon_n].
\end{cases}
\end{align}

The central limit theorem of $\widehat M_n(z)$ is specified below.
\begin{lemma}\label{th2}
Under the conditions of Theorem \ref{th1}, $\widehat M_n(z)$ converges weakly to a two-dimensional Gaussian process $M(\cdot)$ satisfying for $z\in\mathcal{C}$
\begin{align*}
\re M(z)=&\left(d_1(z)-d_2(z)\right)\Bigg\{1-z^{-1}\left[\int\frac{x}{\left({1+xg_{2}(z)}\right)^2}dH_{1}(x)\right]^{-1}\\
&\times\int\frac{x^2}{\left({1+xg_{2}(z)}\right)^2}dH_{1}(x)\int\frac{t}
{\left({1+g_{1}(z)t}\right)^2}dH_{2}(t)\Bigg\}^{-1}
\end{align*}
and for $z_1,z_2\in \mathcal{C}\cup\overline{\mathcal{C}}$ with $\overline{\mathcal{C}}=\{\bar z:z\in\mathcal{C}\}$,
\begin{align*}
{\rm Cov}\left( M(z_1),M(z_2)\right)=2\frac{\partial^2}{\partial z_2\partial z_1}\int_0^{f(z_1,z_2)}\frac{1}{1-z}dz.
\end{align*}

\end{lemma}

\begin{proof}[Proof of Theorem \ref{th1}]
From \cite{Yin1988} and \cite{BaiYin1993}, we conclude that
\begin{align}\label{al7}
 \lambda_{\max}\left(\frac1N\bx_n'\bx_n\right)\to \left(1+\sqrt c\right)^2\quad{\rm a.s.}
\end{align}
and
\begin{align*}
 \lambda_{\min}\left(\frac1N\bx_n'\bx_n\right)\to \left(1-\sqrt c\right)^2\quad{\rm a.s.}
\end{align*}
The upper and lower bounds of the extreme eigenvalues of $\bb_n$ depends largely on the signs of $s_1$ and $s_n$.  Since $s_1>0$, we have
\begin{align*}
\lambda_{\max}(\bb_n)\le s_1 \lambda_{\max}^{\bt_{2n}} \lambda_{\max}\left(\frac1N\bx_n'\bx_n\right)\le s_1 \lambda_{\max}^{\bt_{2n}} \left(1+\sqrt c\right)^2\quad{\rm a.s.}
\end{align*}
If $s_n>0$, then we have
\begin{align*}
\lambda_{\min}(\bb_n)\ge s_n \lambda_{\min}^{\bt_{2n}} I_{(0,1)}(c)\lambda_{\min}\left(\frac1N\bx_n'\bx_n\right)\ge s_n \lambda_{\min}^{\bt_{2n}} I_{(0,1)}(c)\left(1-\sqrt c\right)^2\quad{\rm a.s.}
\end{align*}
Otherwise, we get
\begin{align*}
\lambda_{\min}(\bb_n)\ge s_n \lambda_{\max}^{\bt_{2n}} \lambda_{\max}\left(\frac1N\bx_n'\bx_n\right)\ge s_n \lambda_{\max}^{\bt_{2n}} \left(1+\sqrt c\right)^2\quad{\rm a.s.}
\end{align*}
Combining the definitions of $x_l,x_r$, we find with probability $1$
\begin{align*}
\liminf_{n\to\infty}\min\left(x_r-\lambda_{\max}(\bb_n),\lambda_{\min}(\bb_n)-x_l\right)>0.
\end{align*}
Since $F^{\bb_n}\to F^{c,H_1.H_2}$ with probability $1$ the support of $F^{c_n,H_{1n},H_{2n}}$ is contained in interval (\ref{int}) with probability 1. Thus, by (\ref{al:4}), for $f\in\{f_1,\cdots,f_{\kappa}\}$ and large $n$, with probability $1$,
\begin{align*}
\int f(x)dG_n(x)=-\frac1{2\pi i}\oint f(z)M_{n}(z)dz
\end{align*}
where the complex integral is over ${\mathcal{C}\cup\overline{\mathcal{C}}}$. For $v\in[0,n^{-1}\varepsilon_n]$, note that
\begin{align*}
\left|M_n(x_r+iv)-M_n(x_r+in^{-1}\varepsilon_n)\right|
\le 4n\left|\max\left(\lambda_{\max}(\bb_n), e_r\right)-x_r\right|^{-1}
\end{align*}
and
\begin{align*}
\left|M_n(x_l+iv)-M_n(x_l+in^{-1}\varepsilon_n)\right|
\le 4n\left|\min\left(\lambda_{\min}(\bb_n), e_l\right)-x_l\right|^{-1}.
\end{align*}
It follows that for large $n$, with probability $1$,
\begin{align*}
&\left|\oint f(z)\left(M_{n}(z)-\widehat M_n(z)\right)dz\right|\\
\le &8K\varepsilon_n\left[\left|\max\left(\lambda_{\max}(\bb_n), e_r\right)-x_r\right|^{-1}+\left|\min\left(\lambda_{\min}(\bb_n), e_l\right)-x_l\right|^{-1}\right]\to 0
\end{align*}
where $e_l \ (e_r)$ is the left endpoint (right endpoint) of interval (\ref{int}) and $K$ is the bound on $f$ over $\mathcal{C}$.

Note that the mapping
\begin{align*}
\widehat M_n(\cdot)\rightarrow \left(-\frac1{2\pi i}\oint f_1(z)\widehat M_n(z)dz,\cdots,-\frac1{2\pi i}\oint f_{\kappa}(z)\widehat M_n(z)dz\right)
\end{align*}
is continuous. Using Lemma \ref{th2}, we complete the proof of Theorem \ref{th1}.
\end{proof}

\section{The Gaussian case}

This section is to prove Lemma \ref{th2} under the Gaussian case, i.e., $\{x_{jk}\},j=1,\cdots,N,k=1,\cdots,n$ are standard normal random variables. Since $\bt_{1n}$ is symmetric there exists an orthogonal matrix ${\bf U}$ such that
$$\bt_{1n}={\bf U}{\rm diag}\left(s_1,\cdots,s_n\right){\bf U}'.$$
Note that ${\bf X}_n$ has the same distribution as ${\bf X}_n{\bf U}$. It then suffices to consider
$$\widetilde{\bf B}_n=\frac1N\sum_{k=1}^ns_{k}\bt_{2n}^{1/2}{\bf x}_k{\bf x}_k'\bt_{2n}^{1/2}\triangleq \frac1N\sum_{k=1}^ns_{k}\by_k\by_k'$$
where ${\bf x}_k$ is the $k$-th column of $\bx_n$. In what follows, we omit the symbol $\widetilde{\cdot}$ from the notation of $\widetilde{\bf B}_n$ in order to simplify notation. Rewrite for $z\in\mathcal{C}_n$
\begin{align*}
M_n(z)=N[m_n(z)-\re m_n(z)]+N[\re m_n(z)-m_n^0(z)]\triangleq M_{n1}(z)+M_{n2}(z).
\end{align*}
We below consider the random part $M_{n1}(z)$ and the nonrandom part $M_{n2}(z)$ separately to complete the proof of Lemma \ref{th2}.

In the sequel we assume $x_{jk},j=1,\cdots,N,k=1,\cdots,n$ are truncated at $\delta_n\sqrt n$, centralized and re-normalized. The details are omitted which is similar to Bai and Silverstein (2004).

We start with two probability inequalities for extreme eigenvalues of $\bb_n$. It is well known (see \cite{Yin1988},\cite{bai2004clt}) that for any $l$, $\eta_1>(1+\sqrt c)^2$ and $\eta_2<(1-\sqrt c)^2$
\begin{align*}
{\rm P}\left(\lambda_{\max}\left(\frac1N\bx_n'\bx_n\right)\ge\eta_1\right)=o(n^{-l})
\end{align*}
and
\begin{align*}
{\rm P}\left(\lambda_{\min}\left(\frac1N\bx_n'\bx_n\right)\le\eta_2\right)=o(n^{-l}).
\end{align*}
Thus, letting
\begin{align*}
\eta_r\in
\begin{cases}
(0,x_r),&c\ge1,\\
(\limsup_{n}s_1\lambda_{\max}^{\bt_{2n}}\left(1+\sqrt c\right)^2,x_r),& {\rm otherwise},
\end{cases}
\end{align*}
we have for any $l>0$
\begin{align}\label{max}
{\rm P}\left(\lambda_{\max}\left(\bb_n\right)\ge\eta_r\right)=o(n^{-l}).
\end{align}
Likewise, we have
\begin{align}\label{min}
{\rm P}\left(\lambda_{\min}\left(\bb_n\right)\le\eta_l\right)=o(n^{-l}).
\end{align}
where
\begin{align*}
\eta_l\in
\begin{cases}
(x_l ,0),&c\ge1,\\
(x_l,\liminf_ns_n\lambda_{\min}^{\bt_{2n}}I_{(0,1)}(c)\left(1-\sqrt c\right)^2),&{\rm if} \ \liminf_ns_n\lambda_{\min}^{\bt_{2n}}I_{(0,1)}(c)>0,\\
(x_l,\liminf_ns_n\lambda_{\max}^{\bt_{2n}}\left(1+\sqrt c\right)^2),&{\rm if} \ \liminf_ns_n\lambda_{\min}^{\bt_{2n}}I_{(0,1)}(c)\le0.
\end{cases}
\end{align*}
Here $\eta_l,\eta_r,x_l,x_r$ can be chosen such that
\begin{align}\label{eq18}
x_r-\eta_r>2\tau^2\quad {\rm and} \quad\eta_l-x_l>2\tau^2,
\end{align}
where $\tau$ are the upper bound of the spectral norms of $\bt_{1n}$ and $\bt_{2n}$ defined before.

\subsection{The limiting distribution of $M_{n1}(z)$}

The aim of this part is to find the limiting distribution of $M_{n1}(z)$. That is to say, we show for any positive integer $r$, the sum
  $$\sum_{j=1}^r\alpha_jM_{n1}(z_j) \qquad\Im z_j\neq0$$
converges in distribution to a Gaussian random variable.  We will use the central limit theorem for martingale difference sequences to accomplish the goal. Since
\begin{align*}
\lim_{v_0\downarrow 0}\limsup_{n\to\infty}\re\left|\int_{\mathcal{C}_l\cup\mathcal{C}_r}f(z)M_{n1}(z)dz\right|^2\to0,
\end{align*}
 it suffices to consider $z=u+iv_0\in \mathcal{C}_u$. Introduce
\begin{align}
&\bd(z)=\bb_n(z)-z\bi_N, \ \bd_k(z)=\bd(z)-\frac1Ns_k\by_k\by_k', \notag\\
 &\bb_{nk}=\bb_n-\frac1Ns_k\by_k\by_k', \ g_{2n}(z)=\frac1N\rtr(\bd^{-1}(z)\bt_{2n}),\label{eq4}
\end{align} and
\begin{align}
&\varepsilon_k(z)=\by_k'\bd_k^{-1}(z){\by}_k-\rtr(\bd_k^{-1}(z)\bt_{2n}), \ \gamma_k(z)={\by}_k'\bd_k^{-2}(z){\by}_k-\rtr(\bd_k^{-2}(z)\bt_{2n})\notag\\
&\beta_k(z)=\frac1{1+N^{-1}s_k\by_k'\bd_k^{-1}(z)\by_k}, \ \widetilde{\beta}_k(z)=\frac1{1+N^{-1}s_k\rtr(\bd_k^{-1}(z)\bt_{2n})},\label{eq5}\\
&b_{k}(z)=\frac1{1+N^{-1}s_k\re\rtr(\bd_k^{-1}(z)\bt_{2n})}, \ \psi_k(z)=\frac1{{1+s_k\re g_{2n}(z)}}\label{eq6}.
\end{align}

Note that
\begin{align*}
m_n(z)=\frac1N\rtr\left(\bb_n(z)-z\bi_N\right)^{-1}\triangleq \frac1N\rtr \bd^{-1}(z).
\end{align*}
Let ${\rm E}_{0}(\cdot)$ denote mathematical expectation and ${\rm E}_{k}(\cdot)$ denote conditional expectation with respect to the $\sigma$-field given by ${\bf x}_1,\cdots,{\bf x}_k.$ By the formula
\begin{align}\label{al2}
\left(\boldsymbol\Sigma+q\boldsymbol{\alpha\beta}'\right)^{-1}=\boldsymbol\Sigma^{-1}-\frac{q\boldsymbol\Sigma^{-1}\boldsymbol{\alpha\beta}'\boldsymbol\Sigma^{-1}}
{1+q\boldsymbol{\beta}'
\boldsymbol\Sigma^{-1}\boldsymbol\alpha},
\end{align} we have
\begin{align}
M_{n1}(z)
=&\sum_{k=1}^n\rtr\left\{\re_{k}\bd^{-1}(z)-\re_{k-1}\bd^{-1}(z)\right\}
=\sum_{k=1}^n\left({\rm E}_{k}-{\rm E}_{k-1}\right){\rm tr}\left[\bd(z)^{-1}-\bd_k^{-1}(z)\right]\notag\\
=&-\frac1N\sum_{k=1}^n\left({\rm E}_{k}-{\rm E}_{k-1}\right)s_k\beta_k(z)\by_k'\bd_k^{-2}(z)\by_k\notag\\
=&-\frac1N\sum_{k=1}^n\left({\rm E}_{k}-{\rm E}_{k-1}\right)s_k\beta_k(z)\gamma_k(z)-\frac1N\sum_{k=1}^n\left({\rm E}_{k}-{\rm E}_{k-1}\right)s_k\beta_k(z)\rtr(\bd_k^{-2}(z)\bt_{2n})\notag\\
\triangleq&\mathcal{I}_1+\mathcal{I}_2.\label{eq12}
\end{align}
From the identity
\begin{align}\label{al10}
\beta_k(z)-\widetilde{\beta}_k(z)=-\frac1N s_k \widetilde{\beta}_k(z)\beta_k(z)\varepsilon_k(z),
\end{align}
we have
\begin{align*}
\mathcal{I}_1=&-\frac1N\sum_{k=1}^n{\rm E}_{k}s_k\widetilde{\beta}_k(z)\gamma_k(z)+\frac1{N^2}\sum_{k=1}^n\left({\rm E}_{k}-{\rm E}_{k-1}\right)s_k^2 \widetilde{\beta}_k(z)\beta_k(z) \varepsilon_k(z)\gamma_k(z).
\end{align*}
By Lemma \ref{add} and Lemma \ref{lep1}
\begin{align*}
&\frac1{N^4}\sum_{k=1}^n\re|\left({\rm E}_{k}-{\rm E}_{k-1}\right)s_k^2 \widetilde{\beta}_k(z)\beta_k(z) \varepsilon_k(z)\gamma_k(z)|^2\\
\le&\frac C{N^4}\sum_{k=1}^n\re^{1/2}|\widetilde{\beta}_k(z)\beta_k(z)|^2\re^{1/2}|\varepsilon_k(z)\gamma_k(z)|^4\\
\le&\frac C{N^4}\sum_{k=1}^n\re^{1/4}|\varepsilon_k(z)|^8\re^{1/4}|\gamma_k(z)|^8\le\frac C{N}\to0.
\end{align*}
This implies
\begin{align}\label{eq10}
\mathcal{I}_1=-\frac1N\sum_{k=1}^n{\rm E}_{k}s_k\widetilde{\beta}_k(z)\gamma_k(z)+o_p(1).
\end{align}
 Using the same argument and
\begin{align}\label{al11}
\beta_k(z)-\widetilde{\beta}_k(z)=-\frac1N s_k \widetilde{\beta}_k^2(z)\varepsilon_k(z)+\frac1{N^2}s_k^2\beta_k(z)\widetilde{\beta}_k^2(z)\varepsilon_k^2(z),
\end{align}
 one gets
\begin{align}\label{eq11}
\mathcal{I}_2=\frac1{N^2}\sum_{k=1}^n{\rm E}_{k}s_k^2\widetilde{\beta}_k^2(z)\varepsilon_k(z)\rtr(\bd_k^{-2}(z)\bt_{2n})+o_p(1).
\end{align}
From (\ref{eq12}), (\ref{eq10}), and (\ref{eq11}), we conclude that
\begin{align*}
M_{n1}(z)=&-\frac1N\sum_{k=1}^n{\rm E}_{k}s_k\widetilde{\beta}_k(z)\gamma_k(z)+\frac1{N^2}\sum_{k=1}^n{\rm E}_{k}s_k^2\widetilde{\beta}_k^2(z)\varepsilon_k(z)\rtr(\bd_k^{-2}(z)\bt_{2n})+o_p(1).
\end{align*}
Define
\begin{align*}
h_k(z)=&-\frac1N{\rm E}_{k}s_k\widetilde{\beta}_k(z)\gamma_k(z)+\frac1{N^2}{\rm E}_{k}s_k^2\widetilde{\beta}_k^2(z)\varepsilon_k(z)\rtr(\bd_k^{-2}(z)\bt_{2n})\\
=&-N^{-1}\frac{d}{dz}{\rm E}_{k}s_k\widetilde{\beta}_k(z)\varepsilon_k(z).
\end{align*}
Thus we only need to prove that $\sum_{j=1}^r\alpha_j\sum_{k=1}^nh_k(z_j)=\sum_{k=1}^n\sum_{j=1}^r\alpha_jh_k(z_j)$ converges in distribution to a Gaussian random variable. By Lemma \ref{lep2}, it suffices to verify condition $(i)$ and $(ii)$. It follows from Lemma \ref{add} and Lemma \ref{lep1} that
\begin{align*}
\sum_{k=1}^n\re\left|\sum_{j=1}^r\alpha_jh_k(z_j)\right|^4\le&\frac C{N^4}\sum_{k=1}^n\sum_{j=1}^r\alpha_j^4\bigg[\re^{1/2}|\widetilde{\beta}_k|^8\re^{1/2}|\gamma_k(z_j)|^8\\
&+\re^{1/2}|\widetilde{\beta}_k|^{16}
\re^{1/2}|\varepsilon_k(z_j)|^8\bigg]\le\frac C N\to0
\end{align*}
which implies that conditions $(ii)$ of Lemma \ref{lep2} is satisfied.

The next aim is to find a limit in probability of $$\Phi(z_1,z_2)\triangleq\sum_{k=1}^n\re_{k-1}\left[h_k(z_1)h_k(z_2)\right]$$
 for $z_1,z_2$ with nonzero fixed imaginary parts. It is obvious that
\begin{align*}
\Phi(z_1,z_2)=N^{-2}\frac{\partial^2}{\partial z_2\partial z_1}\sum_{k=1}^n\re_{k-1}\left[\re_k\left(s_k\widetilde{\beta}_k(z_1)\varepsilon_k(z_1)\right)\re_k
\left(s_k\widetilde{\beta}_k(z_2)\varepsilon_k(z_2)\right)\right].
\end{align*}
Due to the analysis on page 571 in \cite{bai2004clt}, it is enough to prove that
$$N^{-2}\sum_{k=1}^n s_k^2\re_{k-1}\left[\re_k\left(\widetilde{\beta}_k(z_1)\varepsilon_k(z_1)\right)\re_k\left(\widetilde{\beta}_k(z_2)\varepsilon_k(z_2)\right)\right]$$
converges in probability to a constant. Similar to (\ref{eq9}) in the appendix, it can be verified that $|\widetilde{\beta}_k(z)|$ and $|b_{k}(z)|$ has the same bound as $\beta_k(z)$. Combining (\ref{max}), (\ref{min}), with (\ref{al2}), we have for $l=1,2$ and suitably large $t$
\begin{align}
&\re|\widetilde{\beta}_k(z_l)-b_{k}(z_l)|^2\label{eq13}\\
\le&\re|\widetilde{\beta}_k(z_l)-b_{k}(z_l)|^2
I(\eta_l\le\lambda_{\min}\le\lambda_{\max}\le\eta_r)+\re|\widetilde{\beta}_k(z_l)-b_{k}(z_l)|^2\notag\\
&\times I(\lambda_{\min}<\eta_l \ {\rm or} \ \lambda_{\max}>\eta_r)\notag\\
\le&\frac C{N^2}\re|\rtr(\bd_k^{-1}(z_l)\bt_{2n})-\re\rtr(\bd_k^{-1}(z_l)\bt_{2n})|^2\notag\\
&+C{N}\re|\rtr(\bd_k^{-1}(z_l)\bt_{2n})-\re\rtr(\bd_k^{-1}(z_l)\bt_{2n})|^2I(\lambda_{\min}<\eta_l \ {\rm or} \ \lambda_{\max}>\eta_r)\notag\\
\le&\frac C{N^2}\re|\sum_{j=1,j\neq k}^n(\re_j-\re_{j-1})\left[\rtr(\bd_k^{-1}(z_l))-\rtr(\bd_{kj}^{-1}(z_l))\right]|^2\notag\\
&+C{N}\re|\sum_{j=1,j\neq k}^n(\re_j-\re_{j-1})\left[\rtr(\bd_k^{-1}(z_l))-\rtr(\bd_{kj}^{-1}(z_l))\right]|^2I(\lambda_{\min}<\eta_l \ {\rm or} \ \lambda_{\max}>\eta_r)\notag\\
\le&\frac C{N^2}\sum_{j=1,j\neq k}^n\re|\rtr(\bd_k^{-1}(z_l))-\rtr(\bd_{kj}^{-1}(z_l))|^2\notag\\
&+C{N^2}\sum_{j=1,j\neq k}^n\re|\rtr(\bd_k^{-1}(z_l))-\rtr(\bd_{kj}^{-1}(z_l))|^2I(\lambda_{\min}<\eta_l \ {\rm or} \ \lambda_{\max}>\eta_r)\notag\\
=&\frac C{N^2}\sum_{j=1,j\neq k}^n\re|\frac{s_j\by_j'\bd_{kj}^{-2}(z_l)\by_j}{N(1+N^{-1}s_j\by_j'\bd_{kj}^{-1}(z_l)\by_j)}|^2\notag\\
&+C{N^2}\sum_{j=1,j\neq k}^n\re|\frac{s_j\by_j'\bd_{kj}^{-2}(z_l)\by_j}{N(1+N^{-1}s_j\by_j'\bd_{kj}^{-1}(z_l)\by_j)}|^2I(\lambda_{\min}<\eta_l \ {\rm or} \ \lambda_{\max}>\eta_r)\notag\\
\le&\frac CN+CN^3{\rm P}(\lambda_{\min}<\eta_l \ {\rm or} \ \lambda_{\max}>\eta_r)\le\frac CN+CN^3n^{-t}\to0\notag
\end{align}
where $\bd_{kj}(z)=\bd_j(z)-\frac1Ns_j\by_j\by_j'$.  From the above inequality, we get
\begin{align*}
&\re\Bigg|N^{-2}\sum_{k=1}^n s_k^2\re_{k-1}\bigg[\re_k\left(\widetilde\beta_k(z_1)\varepsilon_k(z_1)\right)\re_k\left(\widetilde\beta_k(z_2)\varepsilon_k(z_2)\right)\\
&-\re_k\left(b_k(z_1)\varepsilon_k(z_1)\right)\re_k\left(b_k(z_2)\varepsilon_k(z_2)\right)\bigg]\Bigg|\\
\le&CN^{-2}\sum_{k=1}^n\bigg[\re|\re_k\left((\widetilde\beta_k(z_1)-b_k(z_1))\varepsilon_k(z_1)\right)\re_k\left(\widetilde\beta_k(z_2)\varepsilon_k(z_2)\right)|\\
&+\re|\re_k\left(b_k(z_1)\varepsilon_k(z_1)\right)\re_k\left((\widetilde\beta_k(z_2)-b_k(z_2))\varepsilon_k(z_2)\right)|\bigg]\\
\le&CN^{-2}\sum_{k=1}^n\bigg[\re^{1/2}|(\widetilde\beta_k(z_1)-b_k(z_1))\varepsilon_k(z_1)|^2\re^{1/2}|\widetilde\beta_k(z_2)\varepsilon_k(z_2)|^2\\
&+\re^{1/2}|b_k(z_1)\varepsilon_k(z_1)|^2\re^{1/2}|\left((\widetilde\beta_k(z_2)-b_k(z_2))\varepsilon_k(z_2)\right)|^2\bigg]\\
\le&CN^{-1}\sum_{k=1}^n\bigg[\re^{1/2}|\widetilde\beta_k(z_1)-b_k(z_1)|^2+\re^{1/2}|(\widetilde\beta_k(z_2)-b_k(z_2)|^2\bigg]\to0,
\end{align*}
which yields
\begin{align*}
&N^{-2}\sum_{k=1}^n s_k^2\re_{k-1}\bigg[\re_k\left(\widetilde\beta_k(z_1)\varepsilon_k(z_1)\right)\re_k\left(\widetilde\beta_k(z_2)\varepsilon_k(z_2)\right)\\
&-\re_k\left(b_k(z_1)\varepsilon_k(z_1)\right)\re_k\left(b_k(z_2)\varepsilon_k(z_2)\right)\bigg]\xrightarrow{i.p.}0.
\end{align*}
Therefore, our goal is to find the limit in probability of
\begin{align*}
&N^{-2}\sum_{k=1}^n s_k^2 b_{k}(z_1)b_{k}(z_2)\re_{k-1}\left[\re_k\left(\varepsilon_k(z_1)\right)\re_k\left(\varepsilon_k(z_2)\right)\right].
\end{align*}
Using the moments of normal random variables, we have
\begin{align*}
\re_{k-1}\left[\re_k\left(\varepsilon_k(z_1)\right)\re_k\left(\varepsilon_k(z_2)\right)\right]
=2\rtr(\bt_{2n}\re_k\bd_k^{-1}(z_1)\bt_{2n}\re_k\bd_k^{-1}(z_2)).
\end{align*}
Consequently, it suffices to study
\begin{align}\label{goal}
N^{-2}\sum_{k=1}^n s_k^2 b_{k}(z_1)b_{k}(z_2)\rtr(\bt_{2n}\re_k\bd_k^{-1}(z_1)\bt_{2n}\re_k\bd_k^{-1}(z_2)).
\end{align}

Let ${\bf R}_{k}(z)=z\bi-\frac{1}N\sum_{j\neq k}s_j\psi_j(z)\bt_{2n}$,
$$\beta_{jk}(z)=\frac1{1+N^{-1}s_j\by_j'\bd_{jk}^{-1}(z)\by_j}\quad {\rm and} \quad b_{jk}(z)=\frac1{1+N^{-1}s_j\re\rtr(\bd_{jk}^{-1}(z)\bt_{2n})}.$$
Write
\begin{align*}
\bd_k(z_1)+{\bf R}_{k}(z_1)=\frac1N\sum_{j\neq k}s_j\by_j\by_j'-\frac{1}N\sum_{j\neq k}s_j\psi_j(z_1)\bt_{2n}
\end{align*}
which implies that
\begin{align*}
&{\bf R}_{k}^{-1}(z_1)+\bd_k^{-1}(z_1)=\frac1N\sum_{j\neq k}s_j{\bf R}_{k}^{-1}(z_1)\by_j\by_j'\bd_k^{-1}(z_1)-\frac{1}N\sum_{j\neq k}s_j\psi_j(z_1){\bf R}_{k}^{-1}(z_1)\bt_{2n}\bd_k^{-1}(z_1).
\end{align*}
Using the formula
\begin{align}\label{al:2}
\left(\boldsymbol\Sigma+q\boldsymbol{\alpha\beta}'\right)^{-1}\boldsymbol{\alpha}=\frac{\boldsymbol\Sigma^{-1}\boldsymbol{\alpha}}
{1+q\boldsymbol{\beta}'
\boldsymbol\Sigma^{-1}\boldsymbol\alpha},
\end{align}
we have
\begin{align}
{\bf R}_{k}^{-1}(z_1)+\bd_k^{-1}(z_1)
=&\frac1N\sum_{j\neq k}s_j\psi_{j}(z_1){\bf R}_{k}^{-1}(z_1)\left(\by_j\by_j'-\bt_{2n}\right)\bd_{jk}^{-1}(z_1)\label{eq14}\\
&+\frac1N\sum_{j\neq k}s_j\left(\beta_{jk}(z_1)-\psi_j(z_1)\right){\bf R}_{k}^{-1}(z_1)\by_j\by_j'\bd_{jk}^{-1}(z_1)\notag\\
&+\frac1N\sum_{j\neq k}s_j\psi_{j}(z_1){\bf R}_{k}^{-1}(z_1)\bt_{2n}\left(\bd_{jk}^{-1}(z_1)-\bd_k^{-1}(z_1)\right)\notag\\
\triangleq&{\bf A_1}(z_1)+{\bf A_2}(z_1)+{\bf A_3}(z_1)\notag.
\end{align}
By a direct calculation, we have for any positive number $t\ge0$
\begin{align*}
\Im\left(z-\frac{1}N\sum_{j\neq k}s_j\psi_j(z)t\right)=&v_0-\frac{1}N\sum_{j\neq k}\frac{s_j^2 t}{|1+s_j\re g_{2n}(z)|^2}\Im\re g_{2n}(\bar z)\\
=&v_0\left(1+\frac{1}{N^2}\sum_{j\neq k}\frac{s_j^2 t}{|1+s_j\re g_{2n}(z)|^2}\re \rtr\left(\bd^{-1}(z)\bd^{-1}(\bar z)\bt_{2n}\right)\right)\ge v_0
\end{align*}
which yields
\begin{align*}
\left\|{\bf R}_{k}^{-1}(z)\right\|\le\frac1{v_0}.
\end{align*}

Let ${\bf M}$ be a $N\times N$ matrix with a nonrandom bound on the spectral norm of ${\bf M}$ for all parameters governing ${\bf M}$ and under all realizations of ${\bf M}$. By the Cauchy-Schwarz inequality, one gets
\begin{align}\label{al4}
\re\bigg|\rtr\big({\bf A_1}(z_1)&{\bf M}\big)\bigg|
\le C\re^{1/2}\Bigg|\by_j'\bd_{jk}^{-1}(z_1){\bf R}_{k}^{-1}(z_1)\by_j\\
&-\rtr\left(
{\bf R}_{k}^{-1}(z_1)\bt_{2n}\bd_{jk}^{-1}(z_1)\right)\Bigg|^2
=O(N^{1/2})\notag.
\end{align}
Let $\widetilde\beta_{jk}(z)=\frac1{1+N^{-1}s_j\rtr(\bd_{jk}^{-1}(z)\bt_{2n})}$. From (\ref{eq13})
\begin{align*}
\re|\widetilde\beta_{jk}(z)-b_{jk}(z)|=O(N^{-1}).
\end{align*}
Applying the above inequality, Lemma \ref{lep1} and Lemma \ref{lep3}, we obtain
\begin{align}\label{al8}
\re|\beta_{jk}(z)-{\psi_j(z)}|^2
\le&C\left[\re|\beta_{jk}(z)-\widetilde\beta_{jk}(z)|^2+|b_{jk}(z)-\psi_j(z)|^2\right]+O(N^{-1})\\
\le&\frac C{N^2}\re^{1/2}|\by_j'\bd_{jk}^{-1}(z)\by_j-\rtr(\bd_{jk}^{-1}(z)\bt_{2n})|^4\notag\\
&+\frac C{N^2}\re|\rtr(\bd_{jk}^{-1}(z)\bt_{2n})-\rtr(\bd^{-1}(z)\bt_{2n})|^2+O(N^{-1})\notag\\
\le&\frac C{N}+\frac C{N^2}+O(N^{-1})=O(N^{-1})\notag
\end{align}
which implies that
\begin{align}\label{al5}
\re\bigg|\rtr\big({\bf A_2}(z_1)&{\bf M}\big)\bigg|
\le\frac CN\sum_{j\neq k}\re^{1/2}\left|\beta_{jk}(z_1)-{\psi_j(z_1)}\right|^2\\
&\times\re^{1/2}\left|\by_j'\bd_{jk}^{-1}(z_1){\bf M}{\bf R}_{k}^{-1}(z_1)\by_j\right|^2
=O(N^{1/2})\notag.
\end{align}
Lemma \ref{lep3} implies that
\begin{align}\label{al6}
\re\left|\rtr\left({\bf A_3}(z_1){\bf M}\right)\right|\le&\frac CN\sum_{j\neq k}\rtr\bigg[\left(\bd_{jk}^{-1}(z_1)-\bd_k^{-1}(z_1)\right)
\times{\bf R}_{k}^{-1}(z_1)\bt_{2n}\bigg]
\le C.
\end{align}
Using (\ref{eq14}), (\ref{al5}), and (\ref{al6}), one gets
\begin{align}\label{eq15}
\rtr(\re_k\left(\bd_k(z_1)\right)\bt_{2n}&\bd_k^{-1}(z_2)\bt_{2n})=-\rtr(\re_k\left({\bf R}_{k}^{-1}(z_1)\right)\bt_{2n}\bd_k^{-1}(z_2)\bt_{2n})\\
&+\rtr(\re_k\left(\ba_1(z_1)\right)\bt_{2n}\bd_k^{-1}(z_2)\bt_{2n})+a(z_1,z_2)\notag
\end{align}
where $\re|a(z_1,z_2)|\le O(N^{1/2})$. Furthermore, write
\begin{align*}
&\rtr(\re_k\left(\ba_1(z_1)\right)\bt_{2n}\bd_k^{-1}(z_2)\bt_{2n})\\
=&-\frac1{N^2}\sum_{j< k}s_j^2\psi_{j}(z_1)\beta_{jk}(z_2)\left[\by_j'\re_k\bd_{jk}^{-1}(z_1)\bt_{2n}\bd_{jk}^{-1}(z_2)\by_j
-\rtr\left(\re_k\bd_{jk}^{-1}(z_1)\bt_{2n}\bd_{jk}^{-1}(z_2)\bt_{2n}\right)\right]\\
&\times\left[\by_j'\bd_{jk}^{-1}(z_2)\bt_{2n}{\bf R}_{k}^{-1}(z_1)\by_j-\rtr\left(\bd_{jk}^{-1}(z_2)\bt_{2n}{\bf R}_{k}^{-1}(z_1)\bt_{2n}\right)\right]\\
&-\frac1{N^2}\sum_{j< k}s_j^2\psi_{j}(z_1)\beta_{jk}(z_2)\left(\by_j'\re_k\bd_{jk}^{-1}(z_1)\bt_{2n}\bd_{jk}^{-1}(z_2)\by_j
-\rtr\left(\re_k\bd_{jk}^{-1}(z_1)\bt_{2n}\bd_{jk}^{-1}(z_2)\bt_{2n}\right)\right)\\
&\times\rtr\left(\bd_{jk}^{-1}(z_2)\bt_{2n}{\bf R}_{k}^{-1}(z_1)\bt_{2n}\right)\\
&-\frac1{N^2}\sum_{j< k}s_j^2\psi_{j}(z_1)\beta_{jk}(z_2)\rtr\left(\re_k\bd_{jk}^{-1}(z_1)\bt_{2n}\bd_{jk}^{-1}(z_2)\bt_{2n}\right)\\
&\times\left[\by_j'\bd_{jk}^{-1}(z_2)\bt_{2n}{\bf R}_{k}^{-1}(z_1)\by_j-\rtr\left(\bd_{jk}^{-1}(z_2)\bt_{2n}{\bf R}_{k}^{-1}(z_1)\bt_{2n}\right)\right]\\
&-\frac1{N^2}\sum_{j< k}s_j^2\psi_{j}(z_1)\beta_{jk}(z_2)\rtr\left(\re_k\bd_{jk}^{-1}(z_1)\bt_{2n}\bd_{jk}^{-1}(z_2)\bt_{2n}\right)
\rtr\left(\bd_{jk}^{-1}(z_2)\bt_{2n}{\bf R}_{k}^{-1}(z_1)\bt_{2n}\right)\\
&+\frac1N\sum_{j< k}s_j\psi_{j}(z_1)\rtr\left[{\bf R}_{k}^{-1}(z_1)\left(\by_j\by_j'-\bt_{2n}\right)\re_k\bd_{jk}^{-1}(z_1)\bt_{2n}\bd_{jk}^{-1}(z_2)\bt_{2n}\right]\\
&-\frac1N\sum_{j< k}s_j\psi_{j}(z_1)\rtr\left[{\bf R}_{k}^{-1}(z_1)\bt_{2n}\re_k\bd_{jk}^{-1}(z_1)\bt_{2n}\left(\bd_k^{-1}(z_2)-\bd_{jk}^{-1}(z_2)\right)\bt_{2n}\right]\\
\triangleq&a_1(z_1,z_2)+a_2(z_1,z_2)+a_3(z_1,z_2)+a_4(z_1,z_2)+a_5(z_1,z_2)+a_6(z_1,z_2).
\end{align*}
It follows from Lemma \ref{add} and Lemma \ref{lep1} that
\begin{align*}
\re|a_1(z_1,z_2)+a_2(z_1,z_2)+a_3(z_1,z_2)+a_5(z_1,z_2)|\le CN^{1/2}.
\end{align*}
In addition, Lemma \ref{lep3} yields that
\begin{align*}
\re|a_6(z_1,z_2)|\le C
\end{align*}
and that
\begin{align*}
\re|a_4(z_1,z_2)+&\frac1{N^2}\sum_{j< k}s_j^2\psi_{j}(z_1)\psi_{j}(z_2)\\
&\rtr\left(\re_k\bd_{k}^{-1}(z_1)\bt_{2n}\bd_{k}^{-1}(z_2)\bt_{2n}\right)
\rtr\left(\bd_{k}^{-1}(z_2)\bt_{2n}{\bf R}_{k}^{-1}(z_1)\bt_{2n}\right)|\le CN^{-1}
\end{align*}
where the last inequality uses (\ref{al8}). (\ref{eq15}), together with the above three inequalities, ensures that
\begin{align*}
&\rtr(\re_k\bd_k^{-1}(z_1)\bt_{2n}\bd_k^{-1}(z_2)\bt_{2n})\bigg[1+\frac1{N^2}\sum_{j< k}s_j^2\psi_{j}(z_1)\psi_{j}(z_2)
\rtr\left(\bd_{k}^{-1}(z_2)\bt_{2n}{\bf R}_{k}^{-1}(z_1)\bt_{2n}\right)\bigg]\\
=&-\rtr\left({\bf R}_{k}^{-1}(z_1)\bt_{2n}\bd_k^{-1}(z_2)\bt_{2n}\right)+a_7(z_1,z_2)
\end{align*}
where $\re|a_7(z_1,z_2)|\le CN^{1/2}$. Combining (\ref{al4}), (\ref{al5}) with (\ref{al6}), one has
\begin{align}\label{eq16}
\rtr&(\re_k\bd_k^{-1}(z_1)\bt_{2n}\bd_k^{-1}(z_2)\bt_{2n})\\
&\times\bigg[1-\frac1{N^2}\sum_{j< k}s_j^2\psi_{j}(z_1)\psi_{j}(z_2)
\rtr\left({\bf R}_{k}^{-1}(z_1)\bt_{2n}{\bf R}_{k}^{-1}(z_2)\bt_{2n}\right)\bigg]\notag\\
=&\rtr\left({\bf R}_{k}^{-1}(z_1)\bt_{2n}{\bf R}_{k}^{-1}(z_2)\bt_{2n}\right)+a_8(z_1,z_2)\notag
\end{align}
where $\re|a_8(z_1,z_2)|\le CN^{1/2}$. From \cite{zlx}
\begin{align*}
g_{2n}(z)\to g_{2}(z) \quad{\rm a.s.}\quad{\rm as} \ n\to\infty.
\end{align*}
It follows that
\begin{align}\label{al9}
|\psi_{j}(z)-\frac1{1+s_jg_{2n}^0(z)}|\le C\left(|\re g_{2n}(z)-g_{2}(z)|+|g_{2n}^0(z)-g_{2}(z)|\right)=o(1),
\end{align}
where $g_{2n}(z)$ is defined at (\ref{eq4}). Note that by (\ref{i4})
\begin{align}
&\frac1N\sum_{j=1}^ns_j^2\frac1{\left(1+s_jg_{2n}^0(z_1)\right)\left(1+s_jg_{2n}^0(z_2)\right)}\notag\\
=&\frac1{g_{2n}^0(z_1)-g_{2n}^0(z_2)}\left[\frac1{N}\sum_{j=1}^n\frac{s_j}{1+s_jg_{2n}^0(z_2)}-\frac1{N}\sum_{j=1}^n\frac{s_j}{1+s_jg_{2n}^0(z_1)}\right]\notag\\
=&\frac{z_1g_{1n}^0(z_1)-z_2g_{1n}^0(z_2)}{g_{2n}^0(z_1)-g_{2n}^0(z_2)}\label{i6}
\end{align}
and by (\ref{i5})
\begin{align}
&\int\frac{t^2}
{\left(1+g_{1n}^0(z_1)t\right)\left(1+g_{1n}^0(z_2)t\right)}dH_{2n}(t)=\frac{z_1g_{2n}^0(z_1)-z_2g_{2n}^0(z_2)}{g_{1n}^0(z_1)-g_{1n}^0(z_2)}.\label{i7}
\end{align}
Using the fact from (\ref{i4}) and (\ref{al9}) that
\begin{align*}
\frac1N\sum_{j\neq k}s_j\psi_j(z)+zg_{1n}^0(z)=o(1),
\end{align*}
we deduce that
\begin{align*}
\rtr\left({\bf R}_{k}^{-1}(z_1)\bt_{2n}{\bf R}_{k}^{-1}(z_2)\bt_{2n}\right)=&\frac N{z_1z_2}\int\frac{t^2}
{\left(1+g_{1n}^0(z_1)t\right)\left(1+g_{1n}^0(z_2)t\right)}dH_{2n}(t)\\
=&\frac N{z_1z_2}\frac{z_1g_{2n}^0(z_1)-z_2g_{2n}^0(z_2)}{g_{1n}^0(z_1)-g_{1n}^0(z_2)}.
\end{align*}

We now deal with $\frac1{N^2}\sum_{j< k}s_j^2\psi_{j}(z_1)\psi_{j}(z_2)$ in (\ref{eq16}). For any $\varepsilon\in(0,1/100)$, we now distinguish the following two cases.
\begin{itemize}
\item[Case 1]: When $k\le n^{1-\varepsilon}$, one gets
\begin{align*}
&\frac1{N^2}\sum_{j< k}\bigg|\bigg[\frac{s_j^2}{(1+s_jg_{2n}^0(z_1))(1+s_jg_{2n}^0(z_2))}-c_n^{-1}\frac{z_1g_{1n}^0(z_1)-z_2g_{1n}^0(z_2)}{g_{2n}^0(z_1)-g_{2n}^0(z_2)}\bigg]\\
&\times\rtr\left({\bf R}_{k}^{-1}(z_1)\bt_{2n}{\bf R}_{k}^{-1}(z_2)\bt_{2n}\right)\bigg|\le CN^{-\varepsilon}=o(1),
\end{align*}
\item[Case 2]: When $k> n^{1-\varepsilon}$, one gets by (\ref{i6})
\begin{align*}
&\frac1{N^2}\bigg|\sum_{j< k}\bigg[\frac{s_j^2}{(1+s_jg_{2n}^0(z_1))(1+s_jg_{2n}^0(z_2))}-c_n^{-1}\frac{z_1g_{1n}^0(z_1)-z_2g_{1n}^0(z_2)}{g_{2n}^0(z_1)-g_{2n}^0(z_2)}\bigg]\\
&\times\rtr\left({\bf R}_{k}^{-1}(z_1)\bt_{2n}{\bf R}_{k}^{-1}(z_2)\bt_{2n}\right)\bigg|\\
\le&\frac1{N}\bigg|\sum_{j< k}\bigg[\frac{s_j^2}{(1+s_jg_{2n}^0(z_1))(1+s_jg_{2n}^0(z_2))}-c_n^{-1}\frac{z_1g_{1n}^0(z_1)-z_2g_{1n}^0(z_2)}{g_{2n}^0(z_1)-g_{2n}^0(z_2)}\bigg]\bigg|\\
=&o(1).
\end{align*}
\end{itemize}
It follows that
\begin{align*}
&\rtr(\re_k\bd_k^{-1}(z_1)\bt_{2n}\bd_k^{-1}(z_2)\bt_{2n})\\
&\times\bigg[1-\frac{k-1}{nz_1z_2}\frac{z_1g_{1n}^0(z_1)-z_2g_{1n}^0(z_2)}{g_{2n}^0(z_1)-g_{2n}^0(z_2)}
\frac{z_1g_{2n}^0(z_1)-z_2g_{2n}^0(z_2)}{g_{1n}^0(z_1)-g_{1n}^0(z_2)}\bigg]\\
=&\frac{N}{z_1z_2}\frac{z_1g_{2n}^0(z_1)-z_2g_{2n}^0(z_2)}{g_{1n}^0(z_1)-g_{1n}^0(z_2)}+a_9(z_1,z_2)
\end{align*}
where $\re|a_9(z_1,z_2)|=o(N)$. Applying Lemma \ref{lep3} and (\ref{al9}), one gets
\begin{align*}
|b_{k}(z)-\frac1{1+s_kg_{2n}^0(z)}|\le \frac CN\re\left|\rtr\left(\bd_k^{-1}(z)-\bd^{-1}(z)\right)\bt_{2n}\right|+o(1)=o(1).
\end{align*}
Set
\begin{align*}
f_{n}(z_1,z_2)=\frac1{z_1z_2}\frac{z_1g_{1n}^0(z_1)-z_2g_{1n}^0(z_2)}
{g_{2n}^0(z_1)-g_{2n}^0(z_2)}\frac{z_1g_{2n}^0(z_1)-z_2g_{2n}^0(z_2)}{g_{1n}^0(z_1)-g_{1n}^0(z_2)}
\end{align*}
and
\begin{align*}
r_{nk}(z_1,z_2)=\frac{s_k^2}{\left(1+s_kg_{2n}^0(z_1)\right)\left(1+s_kg_{2n}^0(z_2)\right)}.
\end{align*}
By (\ref{i6}) and (\ref{i7}), we obtain
\begin{align}\label{eq17}
\left|f_n(z_1,z_2)\right|\le&\left[{c_n}\int\frac{x^2}{|1+g_{2n}^0(z_1)x|^2}dH_{1n}(x)
\int\frac{t^2}{|z_1\left(1+g_{1n}^0(z_1)t)\right)|^2}dH_{2n}(t)\right]^{1/2}\\
&\times\left[{c_n}\int\frac{x^2}{|1+g_{2n}^0(z_2)x|^2}dH_{1n}(x)
\int\frac{t^2}{| z_2 \left(1+g_{1n}^0(z_2)t\right)|^2}dH_{2n}(t)\right]^{1/2}\notag\\
=&\left[\frac{\Im\left(z_1g_{1n}^0(z_1)\right)}{\Im g_{2n}^0(\bar z_1)}
\frac{\Im g_{2n}^0(z_1)-v\int\frac{t}{|z_1\left(1+g_{1n}^0(z_1)t)\right)|^2}dH_{2n}(t)}{\Im\left(\bar z_1g_{1n}^0(\bar z_1)\right)}\right]^{1/2}\notag\\
&\times\left[\frac{\Im\left(z_2g_{1n}^0(z_2)\right)}{\Im g_{2n}^0(\bar z_2)}
\frac{\Im g_{2n}^0(z_2)-v\int\frac{t}{|z_2\left(1+g_{1n}^0(z_2)t)\right)|^2}dH_{2n}(t)}{\Im\left(\bar z_2g_{1n}^0(\bar z_2)\right)}\right]^{1/2}<1.\notag
\end{align}
Using (\ref{eq17}), (\ref{goal}) can be rewritten as for large n
 \begin{align*}
\frac{1}{N z_1z_2}\frac{z_1g_{2n}^0(z_1)-z_2g_{2n}^0(z_2)}{g_{1n}^0(z_1)-g_{1n}^0(z_2)}\sum_{k=1}^n r_{nk}(z_1,z_2)
\left({1-\frac{k-1}{n}f_{n}(z_1,z_2)}\right)^{-1}+o_p(1).
\end{align*}
Applying Lemma \ref{lep4} and (\ref{i6}), we have
 \begin{align*}
&\frac{1}{N}\sum_{k=1}^n r_{nk}(z_1,z_2)\left({1-\frac{k-1}{n}f_{n}(z_1,z_2)}\right)^{-1}\\
=&\left({1-f_{n}(z_1,z_2)}\right)^{-1}\frac{1}{N}\sum_{k=1}^n r_{nk}(z_1,z_2)
-\frac{1}{N}\sum_{k=1}^n \sum_{j=1}^kr_{nj}(z_1,z_2)\\
&\times\left[\frac1{1-{n}^{-1}{k}f_{n}(z_1,z_2)}-\frac1{1-{n}^{-1}{(k-1)}f_{n}(z_1,z_2)}\right]\\
=&\left({1-f_{n}(z_1,z_2)}\right)^{-1}\frac{z_1g_{1n}^0(z_1)-z_2g_{1n}^0(z_2)}
{g_{2n}^0(z_1)-g_{2n}^0(z_2)}
-f_{n}(z_1,z_2)\frac{1}{N}\sum_{k=1}^n \frac{\sum_{j=1}^kr_{nj}(z_1,z_2)}
{k}\\
&\times\frac{{n}^{-1}k}{\left(1-{n}^{-1}{k}f_{n}(z_1,z_2)\right){\left(1-{n}^{-1}{(k-1)}f_{n}(z_1,z_2)\right)}}.
\end{align*}

We next develop the above limit by Abel's lemma. To this end, consider the following two cases, for any $\varepsilon\in(0,1/100)$ and large n.
\begin{itemize}
\item[Case 1]: When $k\le n^{1-\varepsilon}$, one gets
\begin{align*}
&\frac1{N}\sum_{k\le n^{1-\varepsilon}}\bigg|\bigg[\frac{\sum_{j=1}^kr_{nj}(z_1,z_2)}
{k}-c_n^{-1}\frac{z_1g_{1n}^0(z_1)-z_2g_{1n}^0(z_2)}
{g_{2n}^0(z_1)-g_{2n}^0(z_2)}\bigg]\\
&\times\frac{{n}^{-1}kf_{n}(z_1,z_2)}{\left(1-{n}^{-1}{k}f_{n}(z_1,z_2)\right){\left(1-{n}^{-1}{(k-1)}f_{n}(z_1,z_2)\right)}}\bigg|\le CN^{-\varepsilon}=o(1).
\end{align*}
\item[Case 2]: When $k> n^{1-\varepsilon}$, one gets by (\ref{i6})
\begin{align*}
&\frac1{N}\sum_{k\ge n^{1-\varepsilon}}\bigg|\bigg[\frac{\sum_{j=1}^kr_{nj}(z_1,z_2)}
{k}-c_n^{-1}\frac{z_1g_{1n}^0(z_1)-z_2g_{1n}^0(z_2)}
{g_{2n}^0(z_1)-g_{2n}^0(z_2)}\bigg]\\
&\times\frac{{n}^{-1}kf_{n}(z_1,z_2)}{\left(1-{n}^{-1}{k}f_{n}(z_1,z_2)\right){\left(1-{n}^{-1}{(k-1)}f_{n}(z_1,z_2)\right)}}\bigg|\\
\le& \frac C{N}\sum_{k\ge n^{1-\varepsilon}}\left|\frac{\sum_{j=1}^kr_{nj}(z_1,z_2)}
{k}-c_n^{-1}\frac{z_1g_{1n}^0(z_1)-z_2g_{1n}^0(z_2)}
{g_{2n}^0(z_1)-g_{2n}^0(z_2)}\right|=o(1).
\end{align*}
\end{itemize}
Hence, (\ref{goal}) can be transformed into
 \begin{align*}
&f_{n}(z_1,z_2)\left({1-f_{n}(z_1,z_2)}\right)^{-1}-
f_{n}^2(z_1,z_2)\\
&\times\frac{1}{n}\sum_{k=1}^n \frac{{n}^{-1}k}{\left(1-{n}^{-1}{k}f_{n}(z_1,z_2)\right){\left(1-{n}^{-1}{(k-1)}f_{n}(z_1,z_2)\right)}}+o_p(1).
\end{align*}
Thus,
\begin{align*}
(\ref{goal})\xrightarrow{i.p.}&\frac{f(z_1,z_2)}{1-f(z_1,z_2)}-f^2(z_1,z_2)\int_0^1\frac{t}{\left(1-tf(z_1,z_2)\right)^2}dt
=\int_0^{f(z_1,z_2)}\frac{1}{1-z}dz.
\end{align*}
We conclude that
\begin{align*}
\Phi(z_1,z_2)\xrightarrow{i.p.}\frac{\partial^2}{\partial z_2\partial z_1}\int_0^{f(z_1,z_2)}\frac{1}{1-z}dz.
\end{align*}

\subsection{Tightness of $M_{n1}(z)$}

This section is to prove tightness of the sequence of random functions $\widehat M_{n1}(z)$ for $z\in\mathcal{C}$ defined in (\ref{aa}). Similar to Section 3 of Bai and Silverstein (2004) (see \cite{bai2004clt}), it suffices to show  that
\begin{align*}
\sup_{n;z_1,z_2\in\mathcal{C}_n}\frac{\re\left|M_{n1}(z_1)-M_{n1}(z_2)\right|^2}{|z_1-z_2|^2}
\end{align*}
is finite.

We claim that the moments of $\|\bd^{-1}(z)\|$, $\|\bd^{-1}_j(z)\|$, and $\|\bd^{-1}_{jk}(z)\|$ are bounded in n and $z\in\mathcal{C}_n$. Without loss of generality, we only give the proof for $\re\|\bd^{-1}_1(z)\|^p$ and the others are similar.  In fact, it is obvious for $z=u+iv\in\mathcal{C}_u$. For $z\in\mathcal{C}_l$ or $z\in\mathcal{C}_r$, using (\ref{max}) and (\ref{min}), we have for any positive $p$ and suitably large $l$
\begin{align*}
\re\|\bd^{-1}_1(z)\|^p=&\re\|\bd^{-1}_1(z)\|^pI(\eta_l\le\lambda^{\bb_1(z)}\le \eta_r)\\
&+\re\|\bd^{-1}_1(z)\|^pI(\lambda_{\min}^{\bb_1(z)}< \eta_l \ {\rm or} \ \lambda_{\max}^{\bb_1(z)}> \eta_r)\\
\le&\max\{\frac1{|x_r-\eta_r|^p},\frac1{|\eta_l-x_l|^p}\}+v^{-p}{\rm P}(\lambda_{\min}^{\bb_1(z)}< \eta_l \ {\rm or} \ \lambda_{\max}^{\bb_1(z)}> \eta_r)\\
\le&C_1+C_2n^p\varepsilon_n^{-p}n^{-l}\le C_p.
\end{align*}

Write
\begin{align*}
m_n(z_1)-m_n(z_2)=\frac1N\rtr\left(\bd^{-1}(z_1)-\bd^{-1}(z_2)\right)=\frac1N\left(z_1-z_2\right)\rtr\bd^{-1}(z_1)\bd^{-1}(z_2).
\end{align*}
We then have
\begin{align*}
\frac{M_n(z_1)-M_n(z_2)}{z_1-z_2}=&\sum_{j=1}^N\left(\re_j-\re_{j-1}\right)\rtr\bd^{-1}(z_1)\bd^{-1}(z_2)\\
=&\sum_{j=1}^N\left(\re_j-\re_{j-1}\right)\rtr\left(\bd^{-1}(z_1)\bd^{-1}(z_2)-\bd_j^{-1}(z_1)\bd_j^{-1}(z_2)\right)\\
=&\sum_{j=1}^N\left(\re_j-\re_{j-1}\right)\rtr\left(\bd^{-1}(z_1)-\bd_j^{-1}(z_1)\right)\left(\bd^{-1}(z_2)-\bd_j^{-1}(z_2)\right)\\
&+\sum_{j=1}^N\left(\re_j-\re_{j-1}\right)\rtr\left(\bd^{-1}(z_1)-\bd_j^{-1}(z_1)\right)\bd_j^{-1}(z_2)\\
&+\sum_{j=1}^N\left(\re_j-\re_{j-1}\right)\rtr\bd_j^{-1}(z_1)\left(\bd^{-1}(z_2)-\bd_j^{-1}(z_2)\right)\\
=&\frac1{N^2}\sum_{j=1}^Ns_j^2\left(\re_j-\re_{j-1}\right)\beta_j(z_1)\beta_j(z_2)\left(\by_j'\bd_j^{-1}(z_1)\bd_j^{-1}(z_2)\by_j\right)^2\\
&-\frac1{N}\sum_{j=1}^Ns_j\left(\re_j-\re_{j-1}\right)\beta_j(z_1)\by_j'\bd_j^{-2}(z_1)\bd_j^{-1}(z_2)\by_j\\
&-\frac1{N}\sum_{j=1}^Ns_j\left(\re_j-\re_{j-1}\right)\beta_j(z_2)\by_j'\bd_j^{-1}(z_1)\bd_j^{-2}(z_2)\by_j\\
\triangleq&\mathcal{P}_1+\mathcal{P}_2+\mathcal{P}_3.
\end{align*}
Thus, it suffices to show that $\re\left|\mathcal{P}_1+\mathcal{P}_2+\mathcal{P}_3\right|^2$ is bounded. Denote $ \rho_j(z)=\by_j'\bd_j^{-1}(z)\by_j-\re\rtr\left(\bt_{2n}\bd_j^{-1}(z)\right)$. Note that
\begin{align}
\beta_j(z)=&b_j(z)-\frac1Ns_j \beta_j(z)b_j(z)\rho_j(z)\label{eq21}\\
=&b_j(z)-\frac1Ns_j b_j^2(z)\rho_j(z)+\frac1{N^2}s_j^2\beta_j(z)b_j^2(z)\rho_j^2(z)\label{al:1}.
\end{align}
Applying (\ref{eq21}), Lemma \ref{lep1}, and Lemma \ref{lep7}, we deduce for all large $n$
\begin{align*}
|b_j(z)|\le&|\re\beta_j(z)|+\frac1N|s_j b_j(z)\re( \beta_j(z)\rho_j(z))|\\
\le& C_1+C_2|b_j(z)|N^{-1/2}\le \frac{C_1}{1-C_2N^{-1/2}}.
\end{align*}
Hence $|b_j(z)|$ is bounded for all $n$. Using (\ref{eq21}), write
\begin{align*}
\mathcal{P}_1=&\frac1{N^2}\sum_{j=1}^Ns_j^2 b_j(z_1)b_j(z_2)\left(\re_j-\re_{j-1}\right)\left(\by_j'\bd_j^{-1}(z_1)\bd_j^{-1}(z_2)\by_j\right)^2\\
&-\frac1{N^3}\sum_{j=1}^Ns_j^3 b_j(z_1)b_j(z_2)\left(\re_j-\re_{j-1}\right) \beta_j(z_2)\rho_j(z_2)\left(\by_j'\bd_j^{-1}(z_1)\bd_j^{-1}(z_2)\by_j\right)^2\\
&-\frac1{N^3}\sum_{j=1}^Ns_j^3 b_j(z_1)\left(\re_j-\re_{j-1}\right) \beta_j(z_1)\beta_j(z_2)\rho_j(z_1)\left(\by_j'\bd_j^{-1}(z_1)\bd_j^{-1}(z_2)\by_j\right)^2\\
\triangleq&\mathcal{P}_{11}+\mathcal{P}_{12}+\mathcal{P}_{13}.
\end{align*}
By Lemma \ref{lep1}, we deduce that
\begin{align*}
\re|\mathcal{P}_{11}|^2=&\frac1{N^4}\re\Bigg|\sum_{j=1}^Ns_j^2 b_j(z_1)b_j(z_2)\left(\re_j-\re_{j-1}\right)\Big[\left(\by_j'\bd_j^{-1}(z_1)\bd_j^{-1}(z_2)\by_j\right)^2\\
&-\left(\rtr\bd_j^{-1}(z_1)\bd_j^{-1}(z_2)\bt_{2n}\right)^2\Big]\Bigg|^2\\
\le&\frac C{N^4}\sum_{j=1}^N\re\left|\by_j'\bd_j^{-1}(z_1)\bd_j^{-1}(z_2)\by_j-\rtr\bd_j^{-1}(z_1)\bd_j^{-1}(z_2)\bt_{2n}\right|^4\\
&+\frac C{N^2}\sum_{j=1}^N\re\left|\by_j'\bd_j^{-1}(z_1)\bd_j^{-1}(z_2)\by_j-\rtr\bd_j^{-1}(z_1)\bd_j^{-1}(z_2)\bt_{2n}
\right|^2\\
\le&\frac CN+C\le C.
\end{align*}
Using Lemma \ref{lep7} and Lemma \ref{lep1}, one finds
\begin{align*}
\re|\mathcal{P}_{12}|^2=
&\frac1{N^6}\re\left|\sum_{j=1}^Ns_j^3 b_j(z_1)b_j(z_2)\left(\re_j-\re_{j-1}\right) \beta_j(z_2)\rho_j(z_2)\left(\by_j'\bd_j^{-1}(z_1)\bd_j^{-1}(z_2)\by_j\right)^2\right|^2\\
\le&\frac C{N^6}\sum_{j=1}^N\re\left|\beta_j(z_2)\rho_j(z_2)\left(\by_j'\bd_j^{-1}(z_1)\bd_j^{-1}(z_2)\by_j
-\rtr\bd_j^{-1}(z_1)\bd_j^{-1}(z_2)\bt_{2n}\right)^2\right|^2\\
&+\frac C{N^2}\sum_{j=1}^N\re\left|\beta_j(z_2)\rho_j(z_2)\right|^2\\
\le&\frac C{N^6}\sum_{j=1}^N\re^{1/2}\left|\rho_j(z_2)\left(\by_j'\bd_j^{-1}(z_1)\bd_j^{-1}(z_2)\by_j
-\rtr\bd_j^{-1}(z_1)\bd_j^{-1}(z_2)\bt_{2n}\right)^2\right|^4\\
&\times\re^{1/2}\left|\beta_j(z_2)\right|^4+\frac C{N^2}\sum_{j=1}^N\re^{1/2}\left|\beta_j(z_2)\right|^4\re^{1/2}\left|\rho_j(z_2)\right|^4\\
\le&\frac C{N^2}+C\le C.
\end{align*}
By the same argument, we get $\re|\mathcal{P}_{13}|^2\le C$. Hence, we obtain
\begin{align*}
\re|\mathcal{P}_{1}|^2\le C.
\end{align*}
For $\mathcal{P}_{2}$ and $\mathcal{P}_{3}$, we only need to analyze one of them due to their similarity. From (\ref{eq21}), it is obvious that
\begin{align*}
\mathcal{P}_2=&-\frac1{N}\sum_{j=1}^Ns_j b_j(z_1)\left(\re_j-\re_{j-1}\right)\by_j'\bd_j^{-2}(z_1)\bd_j^{-1}(z_2)\by_j\\
&+\frac1{N^2}\sum_{j=1}^Ns_j^2 b_j(z_1)\left(\re_j-\re_{j-1}\right)\beta_j(z_1)\rho_j(z_1)\by_j'\bd_j^{-2}(z_1)\bd_j^{-1}(z_2)\by_j.
\end{align*}
This yields that
\begin{align*}
\re|\mathcal{P}_2|^2=&\frac1{N^2}\re\left|\sum_{j=1}^Ns_j b_j(z_1)\left(\re_j-\re_{j-1}\right)\by_j'\bd_j^{-2}(z_1)\bd_j^{-1}(z_2)\by_j\right|^2\\
&+\frac1{N^4}\re\left|\sum_{j=1}^Ns_j^2 b_j(z_1)\left(\re_j-\re_{j-1}\right)\beta_j(z_1)\rho_j(z_1)\by_j'\bd_j^{-2}(z_1)\bd_j^{-1}(z_2)\by_j\right|^2\\
\le&\frac C{N^2}\sum_{j=1}^N\re\left|\by_j'\bd_j^{-2}(z_1)\bd_j^{-1}(z_2)\by_j-\rtr\bd_j^{-2}(z_1)\bd_j^{-1}(z_2)\bt_{2n}\right|^2\\
&+\frac C{N^4}\sum_{j=1}^N\re\left|\beta_j(z_1)\rho_j(z_1)\by_j'\bd_j^{-2}(z_1)\bd_j^{-1}(z_2)\by_j\right|^2\\
\le& C+\frac C{N^4}\sum_{j=1}^N\re\left|\beta_j(z_1)\rho_j(z_1)\left(\by_j'\bd_j^{-2}(z_1)\bd_j^{-1}(z_2)\by_j
-\rtr\bd_j^{-2}(z_1)\bd_j^{-1}(z_2)\bt_{2n}\right)\right|^2\\
&+\frac C{N^2}\sum_{j=1}^N\re\left|\beta_j(z_1)\rho_j(z_1)\right|^2\le C
\end{align*}
where the first inequality is from Lemma \ref{lep1} and the last inequality is from Lemma \ref{lep7}. Therefore, we conclude that
\begin{align*}
\sup_{n;z_1,z_2\in\mathcal{C}_n}\frac{\re\left|M_{n1}(z_1)-M_{n1}(z_2)\right|^2}{|z_1-z_2|^2}\le\sup_{n;z_1,z_2\in\mathcal{C}_n}
\re\left|\mathcal{P}_1+\mathcal{P}_2+\mathcal{P}_3\right|^2\le C.
\end{align*}
This implies that $\widehat M_{n1}(z)$ is tight.

\subsection{Convergence of $M_{n2}(z)$}

Let ${\bf W}(z)=\frac1N\sum_{j=1}^ns_j\psi_j(z)\bt_{2n}-z\bi$. Our fist aim is to prove that $\left\|{\bf W}^{-1}(z)\right\|$ is uniformly bounded on $\mathcal{C}_n$. Indeed we have for any positive number $t\ge0$
\begin{align*}
\Im\bigg(\frac{1}N\sum_{j=1}^n&s_j\psi_j(z)t -z\bigg)=\frac{1}N\sum_{j=1}^n\frac{s_j^2 t}{|1+s_j\re g_{2n}(z)|^2}\re\Im g_{2n}(\bar z)-v\\
=&-v\left(1+\frac{1}{N^2}\sum_{j\neq k}\re\frac{s_j^2 t}{|1+s_j\re g_{2n}(z)|^2}\re\rtr\left(\bd^{-1}(z)\bd^{-1}(\bar z)\bt_{2n}\right)\right)\le -v
\end{align*}
which yields $\left\|{\bf W}^{-1}(z)\right\|$ is bounded by $v_0^{-1}$ on $\mathcal{C}_u$. Since $\Im\left(zg_1(z)\right)>0$, there exists a positive constant $\delta_1$ such that for any $t$ in the support of $H_2$
$$\inf_{z\in\mathcal{C}_l\cup\mathcal{C}_r}|zg_1(z)t+z|\ge\delta_1.$$
Moreover, since $zg_1(z)$ is continuous on $\mathcal{C}_l\cup\mathcal{C}_r$, there exists $C_0>0$ such that
\begin{align*}
\sup_{z\in\mathcal{C}_l\cup\mathcal{C}_r}|zg_1(z)|<C_0.
\end{align*}
Additionally, using $H_{2n}\xrightarrow{d}H_2$, for all large $n$ and any $t$ in the support of $H_2$, there exists an eigenvalue $\lambda^{\bt_{2n}}$ of $\bt_{2n}$ such that
\begin{align*}
|\lambda^{\bt_{2n}}-t|\le\frac{\delta_1}{4C_0}.
\end{align*}
Assume for the moment that
\begin{align}\label{al:3}
\sup_{z\in\mathcal{C}_l\cup\mathcal{C}_r}|\frac{1}N\sum_{j=1}^n&s_j\psi_j(z)+zg_1(z)|<\frac{\delta_1}{4\tau}.
\end{align}
It follows that
\begin{align*}
&\inf_{z\in\mathcal{C}_l\cup\mathcal{C}_r}|\frac1N\sum_{j=1}^ns_j\psi_j(z)\lambda^{\bt_{2n}}-z|
\ge\inf_{z\in\mathcal{C}_l\cup\mathcal{C}_r}|zg_1(z)t+z|\\
&-
\sup_{z\in\mathcal{C}_l\cup\mathcal{C}_r}|zg_1(z)||\lambda^{\bt_{2n}}-t|
-\sup_{z\in\mathcal{C}_l\cup\mathcal{C}_r}|\lambda^{\bt_{2n}}||\frac{1}N\sum_{j=1}^ns_j\psi_j(z)+zg_1(z)|\ge\delta/2.
\end{align*}
We conclude that
\begin{align}\label{al12}
\sup_{n,\mathcal{C}_n}\left\|{\bf W}^{-1}(z)\right\|<\infty.
\end{align}

We are now in position to prove (\ref{al:3}), i.e.,
\begin{align*}
\sup_{z\in\mathcal{C}_l\cup\mathcal{C}_r}|\frac{1}N\sum_{j=1}^n&s_j\psi_j(z)+zg_1(z)|\to 0\quad {\rm as} \ n\to\infty.
\end{align*}
By (\ref{max}), (\ref{min}) and (\ref{eq18}), we find
\begin{align*}
&\sup_{z\in\mathcal{C}_l\cup\mathcal{C}_r}|\re g_{2n}(z)|=\sup_{z\in\mathcal{C}_l\cup\mathcal{C}_r}\left|\frac1N\re\left[\rtr(\bd^{-1}\bt_{2n})I(\eta_l\le \lambda^{\bb_n}\le \eta_r)\right]\right|\\
&+\sup_{z\in\mathcal{C}_l\cup\mathcal{C}_r}\left|\frac1N\re\left[\rtr(\bd^{-1}\bt_{2n})I(\lambda_{\min}^{\bb_n}<\eta_l \ {\rm or} \  \lambda_{\max}^{\bb_n}>\eta_r)\right]\right|\\
\le&\tau\max\left\{\frac1{x_l-\eta_l},\frac1{\eta_r-x_r}\right\}+\tau n^{1+\alpha}{\rm P}(\lambda_{\min}^{\bb_n}<\eta_l \ {\rm or} \  \lambda_{\max}^{\bb_n}>\eta_r)\\
<& 1/2\tau+1/4\tau=3/4\tau
\end{align*}
which implies that $\psi_j(z)$ is bounded on ${\mathcal{C}_l\cup\mathcal{C}_r}$. For ${z\in\mathcal{C}_l\cup\mathcal{C}_r}$, rewrite
\begin{align*}
\frac{1}N\sum_{j=1}^ns_j\psi_j(z)&+zg_1(z)=(c_n-c)\int\frac{x}{1+x\re g_{2n}(z)}dH_{1n}(x)\\
&-c(\re g_{2n}(z)-g_2(z))\int\frac{x^2}{(1+x\re g_{2n}(z))(1+xg_2(z))}dH_{1n}(x)\\
&+c\int\frac{x}{1+xg_2(z)}d(H_{1n}(x)-dH_{1}(x)).
\end{align*}
Using Lemma \ref{lep10}, one gets
\begin{align}\label{al:7}
\sup_{z\in\mathcal{C}_l\cup\mathcal{C}_r}|\re g_{2n}(z)-g_2(z)|\to0 \quad{\rm as} \ n\to\infty.
\end{align}
Since the functions $x/(1+xg_2(z))$ in $x\in\{s_1,\cdots,s_n\}$ form a bounded, equicontinuous family $z$ ranges in $\mathcal{C}_l\cup\mathcal{C}_r$, by Problem 8, page 17 in \cite{bill}] and the fact that $H_{1n}\to H_1$ we find that
\begin{align}\label{eq19}
\sup_{z\in\mathcal{C}_l\cup\mathcal{C}_r}\left|\int\frac{x}{1+xg_2(z)}d(H_{1n}(x)-dH_{1}(x))\right|\to0
\end{align}
Combining (\ref{al:7}) with (\ref{eq19}) we obtain (\ref{al:3}).

Write $\bd(z)-{\bf W}\left(z\right)=\frac1N\sum_{j=1}^ns_j\by_j\by_j'-\frac1N\sum_{j=1}^ns_j\psi_j(z)\bt_{2n}$. Taking inverses and then expected value, we have
\begin{align}\label{al:6}
&\bw^{-1}(z)-\re\bd^{-1}(z)\\
=&\bw^{-1}(z)\re\left[\frac1N\sum_{j=1}^ns_j\by_j\by_j'\bd^{-1}(z)-\frac1N\sum_{j=1}^ns_j\psi_j(z)\bt_{2n}\bd^{-1}(z)\right]\notag\\
=&\bw^{-1}(z)\re\left[\frac1N\sum_{j=1}^ns_j\beta_j(z)\by_j\by_j'\bd_j^{-1}(z)-\frac1N\sum_{j=1}^ns_j\psi_j(z)\bt_{2n}\bd^{-1}(z)\right]\notag.
\end{align}
Taking the trace on both sides and dividing by $-1$, one obtains
\begin{align}\label{eq20}
d_{n1}(z)
=&-\frac1{N}\sum_{j=1}^ns_j\re\beta_j(z)\left(\by_j'\bd_j^{-1}(z)\bw^{-1}(z)\by_j-\re\rtr\left(\bw^{-1}(z)\bt_{2n}\bd_j^{-1}(z)\right)\right)\\
&-\frac1{N}\sum_{j=1}^ns_j\re\beta_j(z)\left(\re\rtr\left(\bw^{-1}(z)\bt_{2n}\bd_j^{-1}(z)\right)
-\re\rtr\left(\bw^{-1}(z)\bt_{2n}\bd^{-1}(z)\right)\right)\notag\\
&-\frac1{N}\sum_{j=1}^ns_j\re\left(\beta_j(z)-\psi_j(z)\right)\re\left(\rtr\left(\bw^{-1}(z)\bt_{2n}\bd^{-1}(z)\right)\right)\notag\\
\triangleq&\mathcal{J}_1+\mathcal{J}_2+\mathcal{J}_3\notag,
\end{align}
where $d_{n1}(z)=N\left[\re m_n(z)-\int\frac1{\frac1N\sum_{j=1}^ns_j\psi_j(z)x -z}dH_{2n}(x)\right]$.
 From (\ref{al:1}),
 $\mathcal{J}_1$ can be represented as
\begin{align*}
\mathcal{J}_1=&\frac1{N^2}\sum_{j=1}^ns_j^2 b_j^2(z)\re\rho_j(z)\by_j'\bd_j^{-1}(z)\bw^{-1}(z)\by_j\\
&-\frac1{N^3}\sum_{j=1}^ns_j^3 b_j^2\re\beta_j (z)\rho_j^2(z) \left(\by_j'\bd_j^{-1}(z)\bw^{-1}(z)\by_j-\rtr\left(\bw^{-1}(z)\bt_{2n}\bd_j^{-1}(z)\right)\right)\\
&-\frac1{N^3}\sum_{j=1}^ns_j^3 b_j^2\re\beta_j (z)\rho_j^2(z) \left(\rtr\left(\bw^{-1}(z)\bt_{2n}\bd_j^{-1}(z)\right)-\re\rtr\left(\bw^{-1}(z)\bt_{2n}\bd_j^{-1}(z)\right)\right)\\
\triangleq&\mathcal{J}_{11}+\mathcal{J}_{12}+\mathcal{J}_{13}.
\end{align*}
Applying Lemma \ref{lep1}, Lemma \ref{lep7}, and (\ref{al12}), we get
\begin{align*}
\left|\mathcal{J}_{12}\right|\le&\frac C{N^3}\sum_{j=1}^n\left(\re\left|\beta_j (z)\right|^4\right)^{1/4}\left(\re\left|\rho_j(z) \right|^8\right)^{1/4}\\
&\times
\left(\re \left|\by_j'\bd_j^{-1}(z)\bw^{-1}(z)\by_j-\rtr\left(\bw^{-1}(z)\bt_{2n}\bd_j^{-1}(z)\right)\right|^2\right)^{1/2}\\
\le&\frac C{\sqrt N}\to0
\end{align*}
and
\begin{align*}
\left|\mathcal{J}_{13}\right|\le&\frac C{N^3}\sum_{j=1}^n\left(\re\left|\beta_j (z)\right|^4\right)^{1/4}\left(\re\left|\rho_j(z) \right|^8\right)^{1/4}\\
&\times
\left(\re \left|\rtr\left(\bw^{-1}(z)\bt_{2n}\bd_j^{-1}(z)\right)-\re\rtr\left(\bw^{-1}(z)\bt_{2n}\bd_j^{-1}(z)\right)\right|^2\right)^{1/2}\\
\le& \frac CN\to0
\end{align*}
where the last inequality is obtained from Lemma \ref{lep6}. Moreover, we have
\begin{align*}
\mathcal{J}_{11}
=&\frac1{N^2}\sum_{j=1}^ns_j^2 b_j^2(z)\re\varepsilon_j(z)\left(\by_j'\bd_j^{-1}(z)\bw^{-1}(z)\by_j-\rtr\left(\bw^{-1}(z)\bt_{2n}\bd_j^{-1}(z)\right)\right)\\
&+\frac1{N^2}\sum_{j=1}^ns_j^2 b_j^2(z){\rm Cov}\left(\rtr\left(\bt_{2n}\bd_j^{-1}(z)\right),\rtr\left(\bw^{-1}(z)\bt_{2n}\bd_j^{-1}(z)\right)\right)\\
\triangleq&\mathcal{J}_{111}+\mathcal{J}_{112}.
\end{align*}
Using Lemma \ref{lep6} and the Cauchy-Swcharz inequality, one finds $|\mathcal{J}_{112}|\le CN^{-1}$. This yields
\begin{align}\label{eq22}
\mathcal{J}_1=\frac1{N^2}\sum_{j=1}^ns_j^2 b_j^2(z)\re\varepsilon_j(z)\left(\by_j'\bd_j^{-1}(z)\bw^{-1}(z)\by_j-\rtr\left(\bw^{-1}(z)\bt_{2n}\bd_j^{-1}(z)\right)\right)+o(1).
\end{align}
Note that (\ref{eq21}) and
\begin{align*}
&\re\rtr\left(\bw^{-1}(z)\bt_{2n}\bd_j^{-1}(z)\right)
-\re\rtr\left(\bw^{-1}(z)\bt_{2n}\bd^{-1}(z)\right)\\
=&\frac1{N}\sum_{j=1}^ns_j\re\beta_j(z)\by_j'\bd_j^{-1}(z)\bw^{-1}(z)\bt_{2n}\bd_j^{-1}(z)\by_j.
\end{align*}
It follows that
\begin{align*}
\mathcal{J}_{2}=&-\frac1{N^2}\sum_{j=1}^ns_j^2\re\beta_j(z)\re\beta_j(z)\by_j'\bd_j^{-1}(z)\bw^{-1}(z)\bt_{2n}\bd_j^{-1}(z)\by_j\\
=&-\frac1{N^2}\sum_{j=1}^ns_j^2b_j^2(z)\re\by_j'\bd_j^{-1}(z)\bw^{-1}(z)\bt_{2n}\bd_j^{-1}(z)\by_j\\
&+\frac1{N^3}\sum_{j=1}^ns_j^3b_j^2(z)\re\beta_j(z)\rho_j(z)\by_j'\bd_j^{-1}(z)\bw^{-1}(z)\bt_{2n}\bd_j^{-1}(z)\by_j\\
&+\frac1{N^3}\sum_{j=1}^ns_j^3b_j(z)\re\beta_j(z)\rho_j(z)\re\beta_j(z)\by_j'\bd_j^{-1}(z)\bw^{-1}(z)\bt_{2n}\bd_j^{-1}(z)\by_j\\
\triangleq&\mathcal{J}_{21}+\mathcal{J}_{22}+\mathcal{J}_{23}.
\end{align*}
From Lemma \ref{lep1} and (\ref{al12}), we see $|\mathcal{J}_{22}+\mathcal{J}_{23}|\le \frac C{\sqrt N}$. Hence,
\begin{align}\label{eq23}
\mathcal{J}_{2}=&-\frac1{N^2}\sum_{j=1}^ns_j^2b_j^2(z)\re\rtr\left(\bw^{-1}(z)\bt_{2n}\bd_j^{-1}(z)\bt_{2n}\bd_j^{-1}(z)\right)+o(1).
\end{align}
Note that from (\ref{al:1})
\begin{align*}
\re\left(\beta_j(z)-b_j(z)\right)=&\frac1{N^2}s_j^2b_j^3(z)\re\rho_j^2(z)-\frac1{N^3}s_j^3b_j^3(z)\re \beta_j(z)\rho_j^3(z)\\
=&\frac1{N^2}s_j^2b_j^3(z)\re\varepsilon_j^2(z)-\frac1{N^3}s_j^3b_j^3(z)\re \beta_j(z)\rho_j^3(z)\\
&+\frac1{N^2}s_j^2b_j^3(z)
\re\left(\rtr\left(\bd_j^{-1}\bt_{2n}\right)-\re\rtr\left(\bd_j^{-1}\bt_{2n}\right)\right)^2\\
\triangleq&\mathcal{H}_1+\mathcal{H}_2+\mathcal{H}_3.
\end{align*}
By Lemma \ref{lep1}, we obtain
\begin{align*}
|\mathcal{H}_2|\le&\frac C{N^3}\re^{1/2}|\beta_j(z)|^2\re^{1/2}|\rho_j(z)|^6\le CN^{-3/2}=o(N^{-1}).
\end{align*}
Using Lemma \ref{lep6}, we have
\begin{align*}
|\mathcal{H}_3|\le CN^{-2}=o(N^{-1}).
\end{align*}
These imply that
\begin{align*}
\re\left(\beta_j(z)-b_j(z)\right)
=&\frac1{N^2}s_j^2b_j^3(z)\re\varepsilon_j^2(z)+o(N^{-1}).
\end{align*}
Moreover,
\begin{align*}
b_j(z)-\psi_j(z)=&-\frac1{N}s_jb_j(z)\psi_j(z)\re\left(\rtr\left(\bd_j^{-1}\bt_{2n}\right)
-\rtr\left(\bd^{-1}\bt_{2n}\right)\right)\\
=&-\frac1{N^2}s_j^2b_j(z)\psi_j(z)\re\beta_j(z)\by_j'\bd_j^{-1}(z)\bt_{2n}\bd_j^{-1}\by_j\\
=&-\frac1{N^2}s_j^2b_j^2(z)\psi_j(z)\re\by_j'\bd_j^{-1}(z)\bt_{2n}\bd_j^{-1}\by_j\\
&+\frac1{N^3}s_j^3b_j^2(z)\psi_j(z)\re\beta_j(z)\rho_j(z)\by_j'\bd_j^{-1}(z)\bt_{2n}\bd_j^{-1}\by_j.
\end{align*}
From Lemma \ref{lep7} and Lemma \ref{lep1}, we have
\begin{align*}
&\left|\frac1{N^3}s_j^3b_j^2(z)\psi_j(z)\re\beta_j(z)\rho_j(z)\by_j'\bd_j^{-1}(z)\bt_{2n}\bd_j^{-1}\by_j\right|\\
\le&\frac1{N^3}s_j^3\re^{1/2}|\beta_j(z)|^2\re^{1/2}|\rho_j(z)\by_j'\bd_j^{-1}(z)\bt_{2n}\bd_j^{-1}\by_j|^2\le CN^{-3/2}.
\end{align*}
This yields that
\begin{align}\label{al15}
b_j(z)-\psi_j(z)
=&-\frac1{N^2}s_j^2b_j^2(z)\psi_j(z)\re\rtr\bd_j^{-1}(z)\bt_{2n}\bd_j^{-1}\bt_{2n}+o(N^{-1}).
\end{align}
Hence,
\begin{align}\label{al:5}
\re\left(\beta_j(z)-\psi_j(z)\right)
=&-\frac1{N^2}s_j^2b_j^2(z)\psi_j(z)\re\rtr\bd_j^{-1}(z)\bt_{2n}\bd_j^{-1}\bt_{2n}\\
&+\frac1{N^2}s_j^2b_j^3(z)\re\varepsilon_j^2(z)+o(N^{-1})\notag.
\end{align}
Thus, we get
\begin{align}\label{eq24}
\mathcal{J}_3
=&-\frac1{N^3}\sum_{j=1}^ns_j^3b_j^3(z)\re\varepsilon_j^2(z)\re\left(\rtr\left(\bw^{-1}(z)\bt_{2n}\bd^{-1}(z)\right)\right)\\
&+\frac1{N^3}s_j^3b_j^2(z)\psi_j(z)\re\rtr\bd_j^{-1}(z)\bt_{2n}\bd_j^{-1}\bt_{2n}
\re\left(\rtr\left(\bw^{-1}(z)\bt_{2n}\bd^{-1}(z)\right)\right)+o(1)\notag.
\end{align}
From (\ref{eq20}), (\ref{eq22}), (\ref{eq23}) and (\ref{eq24}), we conclude that
\begin{align*}
d_{n1}(z)
=&\frac1{N^2}\sum_{j=1}^ns_j^2 b_j^2(z)\re\varepsilon_j(z)\left(\by_j'\bd_j^{-1}(z)\bw^{-1}(z)\by_j-\rtr\left(\bw^{-1}(z)\bt_{2n}\bd_j^{-1}(z)\right)\right)\\
&-\frac1{N^2}\sum_{j=1}^ns_j^2b_j^2(z)\re\rtr\left(\bw^{-1}(z)\bt_{2n}\bd_j^{-1}(z)\bt_{2n}\bd_j^{-1}(z)\right)\\
&-\frac1{N^3}\sum_{j=1}^ns_j^3\psi_jb_j^2(z)\re\varepsilon_j^2(z)\re\left(\rtr\left(\bw^{-1}(z)\bt_{2n}\bd^{-1}(z)\right)\right)\\
&+\frac1{N^3}s_j^3b_j^2(z)\psi_j(z)\re\rtr\bd_j^{-1}(z)\bt_{2n}\bd_j^{-1}\bt_{2n}\re\left(\rtr\left(\bw^{-1}(z)\bt_{2n}\bd^{-1}(z)\right)\right)+o(1).
\end{align*}
It is evident from (\ref{al15}) that
\begin{align}\label{al:9}
|b_j(z)-\psi_j(z)|\le &\frac C{N}\re\left\|\bd_j^{-1}(z)\bt_{2n}\bd_j^{-1}\bt_{2n}\right\|+o(N^{-1})
\le \frac C{N}.
\end{align}
Then,
\begin{align*}
d_{n1}(z)
=&\frac1{N^2}\sum_{j=1}^ns_j^2 \psi_j^2(z)\re\varepsilon_j(z)\left(\by_j'\bd_j^{-1}(z)\bw^{-1}(z)\by_j-\rtr\left(\bw^{-1}(z)\bt_{2n}\bd_j^{-1}(z)\right)\right)\\
&-\frac1{N^2}\sum_{j=1}^ns_j^2\psi_j^2(z)\re\rtr\left(\bw^{-1}(z)\bt_{2n}\bd_j^{-1}(z)\bt_{2n}\bd_j^{-1}(z)\right)\\
&-\frac1{N^3}\sum_{j=1}^ns_j^3\psi_j^3(z)\re\varepsilon_j^2(z)\re\left(\rtr\left(\bw^{-1}(z)\bt_{2n}\bd^{-1}(z)\right)\right)\\
&+\frac1{N^3}s_j^3\psi_j^3(z)\re\rtr\bd_j^{-1}(z)\bt_{2n}\bd_j^{-1}\bt_{2n}\re\left(\rtr\left(\bw^{-1}(z)\bt_{2n}\bd^{-1}(z)\right)\right)+o(1).
\end{align*}
Considering the moments of the Gaussian variables, we have
\begin{align*}
d_{n1}(z)
=&\frac1{N^2}\sum_{j=1}^ns_j^2\psi_{j}^2(z)\re\rtr\left(\bw^{-1}(z)\bt_{2n}\bd_j^{-1}(z)\bt_{2n}\bd_j^{-1}(z)\right)\\
&-\frac1{N^3}\sum_{j=1}^ns_j^3\psi_{j}^3(z)\re\rtr\left(\bt_{2n}\bd_j^{-1}(z)\bt_{2n}\bd_j^{-1}(z)\right)
\re\left(\rtr\left(\bw^{-1}(z)\bt_{2n}\bd^{-1}(z)\right)\right)+o(1).
\end{align*}

Write $M_{n2}(z)$ as
\begin{align*}
&N\left[\re m_n(z)-m_n^0(z)\right]\\
=&d_{n1}(z)+N\left[\int\frac1{\frac1N\sum_{j=1}^ns_j\psi_j(z)x- z}dH_{2n}(x)+z^{-1}\int\frac1{1+g_{1n}^0(z)x}dH_{2n}(x)\right]\\
=&d_{n1}(z)-N\left(\re g_{2n}(z)-g_{2n}^0(z)\right)\frac1N\sum_{j=1}^n\frac{s_j^2\psi_j(z)}{1+g_{2n}^0(z)s_j}\\
&\times\int\frac{x}
{\left(\frac1N\sum_{j=1}^ns_j\psi_j(z)x- z\right)\left({z+zg_{1n}^0(z)x}\right)}dH_{2n}(x).
\end{align*}
Below we first find the relation between $\left(\re m_n(z)-m_n^0(z)\right)$ and $\left(\re g_{2n}(z)-g_{2n}^0(z)\right)$. Write $\bd(z)+z\bi_N=\frac1N\sum_{k=1}^ns_k\by_k\by_k'$. Multiplying by $\bd^{-1}(z)$ on the right-hand side and using the formula (\ref{al2}), we obtain
\begin{align*}
\bi_N+z\bd^{-1}(z)=&\frac1N\sum_{k=1}^ns_k\by_k\by_k'\left(\bd_k^{-1}(z)-\frac1Ns_k\beta_k(z)\bd_k^{-1}(z)\by_k\by_k'\bd_k^{-1}(z)\right)\\
=&\frac1N\sum_{k=1}^n s_k\beta_k(z)\by_k\by_k'\bd_k^{-1}(z).
\end{align*}
Taking the trace on both side and dividing by $N$, one gets
\begin{align*}
1+zm_n(z)=c_n-c_nn^{-1}\sum_{k=1}^n\beta_k(z).
\end{align*}
Together with (\ref{al1}), we have
\begin{align}\label{al:8}
\underline m_n(z)=-\frac1{zn}\sum_{k=1}^n\beta_k(z).
\end{align}
It is obtained from (\ref{al:5}) and (\ref{al:9})
\begin{align*}
\re\underline m_n(z)+\frac1{zn N^2}\sum_{j=1}^ns_j^2\psi_j^3(z)\re\rtr\left(\bd_j^{-1}(z)\bt_{2n}\bd_j^{-1}(z)\bt_{2n}\right)=-\frac1{zn}\sum_{j=1}^n\psi_j(z)+o(N^{-1}).
\end{align*}
Thus
\begin{align*}
\re\underline m_n(z)-\underline m_n^0(z)=&-\frac1{zn N^2}\sum_{j=1}^ns_j^2\psi_j^3(z)\re\rtr\left(\bd_j^{-1}(z)\bt_{2n}\bd_j^{-1}(z)\bt_{2n}\right)\\
&+\left(\re g_{2n}(z)-g_{2n}^0(z)\right)\frac1{zn}\sum_{j=1}^n\frac{s_j\psi_j(z)}{1+g_{2n}^0(z)s_j}+o(N^{-1}).
\end{align*}
Consequently,
\begin{align*}
\re g_{2n}(z)-g_{2n}^0(z)=&\left[\frac1{zn}\sum_{j=1}^n\frac{s_j\psi_j(z)}{1+g_{2n}^0(z)s_j}\right]^{-1}\times\Bigg[\re\underline m_n(z)-\underline m_n^0(z)\\
&+\frac1{zn N^2}\sum_{j=1}^ns_j^2\psi_j^3(z)\re\rtr\left(\bd_j^{-1}(z)\bt_{2n}\bd_j^{-1}(z)\bt_{2n}\right)\Bigg]+o(N^{-1}).
\end{align*}
Combining the above equalities with $M_{n2}(z)=N\left[\re m_n(z)-m_n^0(z)\right]=n\left[\re\underline m_n(z)-\underline m_n^0(z)\right]$, we conclude that
\begin{align}
&N\left[\re m_n(z)-m_n^0(z)\right]\label{al:10}\\
=&d_{n1}(z)-n\left[\re\underline m_n(z)-\underline m_n^0(z)\right]\left[\frac1{zn}\sum_{j=1}^n\frac{s_j\psi_j(z)}{1+g_{2n}^0(z)s_j}\right]^{-1}\notag\\
&\times\frac1n\sum_{j=1}^n\frac{s_j^2\psi_j(z)}{1+g_{2n}^0(z)s_j}\int\frac{x}
{\left(\frac1N\sum_{j=1}^ns_j\psi_j(z)x- z\right)\left({z+zg_{1n}^0(z)x}\right)}dH_{2n}(x)\notag\\
&-\frac1{z N^2}\sum_{j=1}^ns_j^2\psi_j^3(z)\re\rtr\left(\bd_j^{-1}(z)\bt_{2n}\bd_j^{-1}(z)\bt_{2n}\right)
\left[\frac1{zn}\sum_{j=1}^n\frac{s_j\psi_j(z)}{1+g_{2n}^0(z)s_j}\right]^{-1}\notag\\
&\times\frac1n\sum_{j=1}^n\frac{s_j^2\psi_j(z)}{1+g_{2n}^0(z)s_j}\int\frac{x}
{\left(\frac1N\sum_{j=1}^ns_j\psi_j(z)x- z\right)\left({z+zg_{1n}^0(z)x}\right)}dH_{2n}(x)\notag\\
=&\left(d_{n1}(z)-d_{n2}(z)\right)\Bigg\{1+z\left[\frac1{n}\sum_{j=1}^n\frac{s_j\psi_j(z)}{1+g_{2n}^0(z)s_j}\right]^{-1}
\frac1n\sum_{j=1}^n\frac{s_j^2\psi_j(z)}{1+g_{2n}^0(z)s_j}\label{eq25}\\
&\times\int\frac{x}
{\left(\frac1N\sum_{j=1}^ns_j\psi_j(z)x- z\right)\left({z+zg_{1n}^0(z)x}\right)}dH_{2n}(x)\Bigg\}^{-1}\notag
\end{align}
where
\begin{align*}
d_{n2}(z)=&\frac1{ N^2}\sum_{j=1}^ns_j^2\psi_j^3(z)\re\rtr\left(\bd_j^{-1}(z)\bt_{2n}\bd_j^{-1}(z)\bt_{2n}\right)
\left[\frac1{n}\sum_{j=1}^n\frac{s_j\psi_j(z)}{1+g_{2n}^0(z)s_j}\right]^{-1}\\
&\times\frac1n\sum_{j=1}^n\frac{s_j^2\psi_j(z)}{1+g_{2n}^0s_j}\int\frac{x}
{\left(\frac1N\sum_{j=1}^ns_j\psi_j(z)x- z\right)\left({z+zg_{1n}^0(z)x}\right)}dH_{2n}(x).
\end{align*}
Write
\begin{align*}
&\frac1{n}\sum_{j=1}^n\frac{s_j}{\left(1+g_{2}(z)s_j\right)^2}-\frac zn\sum_{j=1}^n\frac{s_j^2}{\left(1+g_{2}(z)s_j\right)^2}\int\frac{x}
{\left({z+zg_{1}(z)x}\right)^2}dH_{2}(x)\\
=&-\frac{zg_1(z)}{c_n}-\frac 1n\sum_{j=1}^n\frac{s_j^2}{\left(1+g_{2}(z)s_j\right)^2}\left[g_2(z)+z\int\frac{x}
{\left({z+zg_{1}(z)x}\right)^2}dH_{2}(x)\right]\\
=&-\frac{zg_1(z)}{c_n}\left\{1-\frac 1{z^2N}\sum_{j=1}^n\frac{s_j^2}{\left(1+g_{2}(z)s_j\right)^2}\int\frac{x^2}
{\left({1+g_{1}(z)x}\right)^2}dH_{2}(x)\right\}.
\end{align*}
Note that for all $z=u+iv\in\mathcal{C}$
\begin{align*}
&\left|\frac 1{z^2N}\sum_{j=1}^n\frac{s_j^2}{\left(1+g_{2}(z)s_j\right)^2}\int\frac{x^2}
{\left({1+g_{1}(z)x}\right)^2}dH_{2}(x)\right|\\
\le&\frac 1N\sum_{j=1}^n\frac{s_j^2}{\left|1+g_{2}(z)s_j\right|^2}\int\frac{x^2}
{\left|{z+zg_{1}(z)x}\right|^2}dH_{2}(x)\\
=&\frac{\Im \left(zg_1(z)\right)}{\Im g_2(z)}\frac{\Im g_2(z)-v\int\frac{x}
{\left|{z+zg_{1}(z)x}\right|^2}dH_{2}(x)}{\Im \left(zg_1(z)\right)}<1.
\end{align*}
By continuity, we have the denominator of (\ref{eq25}) is bounded away from zero.

We are now in position to find the limits of $d_{n1}(z)$ and $d_{n2}(z)$. Due to (\ref{al2}) and (\ref{al12}), we see that
\begin{align*}
&\left|\re\rtr\left(\bw^{-1}(z)\bt_{2n}\bd_j^{-1}(z)\bt_{2n}\bd_j^{-1}(z)\right)-\re\rtr\left(\bw^{-1}(z)\bt_{2n}\bd^{-1}(z)\bt_{2n}\bd^{-1}(z)\right)\right|\\
\le&\left|\re\rtr\left(\bw^{-1}(z)\bt_{2n}\bd_j^{-1}(z)\bt_{2n}\bd_j^{-1}(z)\right)-\re\rtr\left(\bw^{-1}(z)\bt_{2n}\bd_j^{-1}(z)\bt_{2n}\bd^{-1}(z)\right)\right|\\
&+\left|\re\rtr\left(\bw^{-1}(z)\bt_{2n}\bd^{-1}(z)\bt_{2n}\bd_j^{-1}(z)\right)-\re\rtr\left(\bw^{-1}(z)\bt_{2n}\bd^{-1}(z)\bt_{2n}\bd^{-1}(z)\right)\right|\\
\le&\frac{\tau}N\re|\beta_j(z)|\left|{\bf y}_j'\bd_j^{-1}(z)\bw^{-1}(z)\bt_{2n}\bd_j^{-1}(z)\bt_{2n}\bd_j^{-1}(z){\bf y}_j\right|\\
&+\frac{\tau}N\re|\beta_j(z)|\left|{\bf y}_j'\bd_j^{-1}(z)\bw^{-1}(z)\bt_{2n}\bd^{-1}(z)\bt_{2n}\bd_j^{-1}(z){\bf y}_j\right|\\
\le &C
\end{align*}
where the last inequality is from Lemma \ref{lep1} and Lemma \ref{lep7}. By the same argument, it follows that
\begin{align*}
&\left|\re\rtr\left(\bt_{2n}\bd_j^{-1}(z)\bt_{2n}\bd_j^{-1}(z)\right)-\re\rtr\left(\bt_{2n}\bd^{-1}(z)\bt_{2n}\bd^{-1}(z)\right)\right|\le C.
\end{align*}
Hence,
\begin{align*}
d_{n1}(z)
=&\frac1{N^2}\sum_{j=1}^ns_j^2\psi_{j}^2(z)\re\rtr\left(\bw^{-1}(z)\bt_{2n}\bd^{-1}(z)\bt_{2n}\bd^{-1}(z)\right)\\
&-\frac1{N^3}\sum_{j=1}^ns_j^3\psi_{j}^3(z)\re\rtr\left(\bt_{2n}\bd^{-1}(z)\bt_{2n}\bd^{-1}(z)\right)
\re\left(\rtr\left(\bw^{-1}(z)\bt_{2n}\bd^{-1}(z)\right)\right)+o(1).
\end{align*}

Our next goal is to find the limit of $\re\rtr\left(\bt_{2n}\bd^{-1}(z)\bt_{2n}\bd^{-1}(z)\right)$, $\re\rtr\left(\bw^{-1}(z)\bt_{2n}\bd^{-1}(z)\right)$ and $\re\rtr\left(\bw^{-1}(z)\bt_{2n}\bd^{-1}(z)\bt_{2n}\bd^{-1}(z)\right)$. From (\ref{al:6}), we have
\begin{align}\label{al18}
{\bf W}^{-1}(z)-\bd^{-1}(z)
=&\frac1N\sum_{j=1}^ns_j\psi_j(z){\bf W}^{-1}(z)\left(\by_j\by_j'-\bt_{2n}\right)\bd_{j}^{-1}(z)\\
&+\frac1N\sum_{j=1}^ns_j\left(\beta_{j}(z)-\psi_j(z)\right){\bf W}^{-1}(z)\by_j\by_j'\bd_{j}^{-1}(z)\notag\\
&+\frac1N\sum_{j=1}^ns_j\psi_j(z){\bf W}^{-1}(z)\bt_{2n}\left(\bd_{j}^{-1}(z)-\bd^{-1}(z)\right)\notag\\
\triangleq&{\bf G_1}(z)+{\bf G_2}(z)+{\bf G_3}(z)\notag.
\end{align}

Let ${\bf M}$ be $N\times N$ matrix with a nonrandom bound on the spectral norm of ${\bf M}$ for all parameters governing ${\bf M}$ and under all realizations of ${\bf M}$. Applying Lemma \ref{lep1}, Lemma \ref{lep6}, and Lemma \ref{lep7}, we obtain
\begin{align}\label{al16}
&\re|\beta_{j}(z)-\psi_j(z)|^2
\le\frac C{N^2}\re|\beta_j(z)\left(\by_j'\bd_j^{-1}(z)\by_j-\re\bd^{-1}(z)\bt_{2n}\right)|^2=O(N^{-1})
\end{align}
which implies that
\begin{align}\label{al19}
\re\bigg|\rtr\big({\bf G_2}(z)&{\bf M}\big)\bigg|
\le\frac CN\sum_{j=1}^n\re^{1/2}\left|\beta_{j}(z)-\psi_j(z)\right|^2\\
&\times\re^{1/2}\left|\by_j'\bd_{j}^{-1}(z){\bf M}{\bf W}^{-1}(z)\by_j\right|^2
=O(N^{1/2})\notag.
\end{align}
Form Lemma \ref{lep1} and (\ref{al2}), we have
\begin{align}\label{al20}
\re\left|\rtr\left({\bf G_3}(z){\bf M}\right)\right|\le&\frac C{N^2}\sum_{j=1}^n\left|\re\by_j'\bd_j^{-1}(z){\bf W}^{-1}(z)\bt_{2n}\bd_{j}^{-1}(z)\by_j\right|
\le C.
\end{align}
Furthermore, write
\begin{align*}
&\rtr(\left({\bf G}_1(z)\right)\bt_{2n}\bd^{-1}(z){\bf M})\\
=&\frac1N\sum_{j=1}^ns_j\psi_j(z)\rtr{\bf W}^{-1}(z)\by_j\by_j'\bd_{j}^{-1}(z)\bt_{2n}\left(\bd^{-1}(z)-\bd_{j}^{-1}(z)\right){\bf M}\\
&+\frac1N\sum_{j=1}^ns_j\psi_j(z)\rtr{\bf W}^{-1}(z)\left(\by_j\by_j'\bd_{j}^{-1}(z)\bt_{2n}\bd_j^{-1}(z){\bf M}-\bt_{2n}\bd_{j}^{-1}(z)\bt_{2n}\bd_j^{-1}(z){\bf M}\right)\\
&+\frac1N\sum_{j=1}^ns_j\psi_j(z)\rtr{\bf W}^{-1}(z)\bt_{2n}\bd_{j}^{-1}(z)\bt_{2n}\left(\bd_j^{-1}(z)-\bd^{-1}(z)\right){\bf M}\\
=&-\frac1{N^2}\sum_{j=1}^ns_j^2\psi_j(z)\beta_j(z)\by_j'\bd_{j}^{-1}(z)\bt_{2n}\bd_j^{-1}(z)\by_j\by_j'\bd_{j}^{-1}(z){\bf M}{\bf W}^{-1}(z)\by_j\\
&+\frac1N\sum_{j=1}^ns_j\psi_j(z)\rtr{\bf W}^{-1}(z)\left(\by_j\by_j'\bd_{j}^{-1}(z)\bt_{2n}\bd_j^{-1}(z){\bf M}-\bt_{2n}\bd_{j}^{-1}(z)\bt_{2n}\bd_j^{-1}(z){\bf M}\right)\\
&+\frac1{N^2}\sum_{j=1}^ns_j^2\psi_j(z)\beta_j(z)\by_j'\bd_{j}^{-1}(z){\bf M}{\bf W}^{-1}(z)\bt_{2n}\bd_{j}^{-1}(z)\bt_{2n}\bd_j^{-1}(z)\by_j\\
\triangleq&p_1(z)+p_2(z)+p_3(z).
\end{align*}
It is obvious that $\re p_2(z)=0$. Using Lemma \ref{lep1} and Lemma \ref{lep7}, we have
\begin{align*}
\re|p_3(z)|\le C.
\end{align*}
Together with (\ref{al16}) and Lemma \ref{lep1}, one gets
\begin{align*}
\re p_1(z)=-\frac1{N^2}\sum_{j=1}^ns_j^2\psi_j^2(z)\re\by_j'\bd_{j}^{-1}(z)\bt_{2n}\bd_j^{-1}(z)\by_j\by_j'\bd_{j}^{-1}(z){\bf M}{\bf W}^{-1}(z)\by_j+O(1).
\end{align*}
By the proof of Lemma \ref{lep10},  we obtain
\begin{align*}
|\psi_{j}(z)-\frac1{1+s_jg_{2n}^0(z)}|=o(1).
\end{align*}
Let $q_j=\frac{s_j^2}{\left({1+s_jg_{2n}^0(z)}\right)^2}$. Then, combining Lemma \ref{lep1}, we find
\begin{align*}
\re p_1(z)
=&-\frac1{N^2}\sum_{j=1}^nq_j\re\by_j'\bd_{j}^{-1}(z)\bt_{2n}\bd_j^{-1}(z)\by_j\by_j'\bd_{j}^{-1}(z){\bf M}{\bf W}^{-1}(z)\by_j+o(N)\notag\\
=&-\frac1{N^2}\sum_{j=1}^nq_j\re\left(\rtr\bd_{j}^{-1}(z)\bt_{2n}\bd_{j}^{-1}(z)\bt_{2n}\right)
\left(\rtr\bd_{j}^{-1}(z){\bf M}{\bf W}^{-1}(z)\bt_{2n}\right)+o(N)\notag\\
=&-\frac1{N^2}\sum_{j=1}^nq_j\re\left(\rtr\bd^{-1}(z)\bt_{2n}\bd^{-1}(z)\bt_{2n}\right)
\left(\rtr\bd^{-1}(z){\bf M}{\bf W}^{-1}(z)\bt_{2n}\right)+o(N)\notag\\
=&-\frac1{N^2}\sum_{j=1}^nq_j\left(\re\rtr\bd^{-1}(z)\bt_{2n}\bd^{-1}(z)\bt_{2n}\right)
\left(\re\rtr\bd^{-1}(z){\bf M}{\bf W}^{-1}(z)\bt_{2n}\right)+o(N).
\end{align*}
It follows that
\begin{align}\label{al21}
&\re\rtr(\left({\bf G}_1(z)\right)\bt_{2n}\bd^{-1}(z){\bf M})\\
=&-\frac1{N^2}\sum_{j=1}^nq_j\left(\re\rtr\bd^{-1}(z)\bt_{2n}\bd^{-1}(z)\bt_{2n}\right)
\left(\re\rtr\bd^{-1}(z){\bf M}{\bf W}^{-1}(z)\bt_{2n}\right)+o(N)\notag.
\end{align}
From (\ref{al18}), (\ref{al19})-(\ref{al21}), and
\begin{align*}
&\re\rtr(\left({\bf G}_1(z)\right)\bt_{2n}{\bf W}^{-1}(z)\bt_{2n}=0,
\end{align*}
one has
\begin{align}\label{al22}
\frac1N\re\rtr\bd^{-1}(z)\bt_{2n}{\bf W}^{-1}(z)\bt_{2n}=&\frac1N\rtr{\bf W}^{-1}(z)\bt_{2n}{\bf W}^{-1}(z)\bt_{2n}+o(1)\notag\\
=&\int\frac{t^2}{\left(\frac1N\sum_{j=1}^ns_j\psi_j(z)t-z\right)^2}dH_{2n}(t)+o(1).
\end{align}
By the same argument, we get
\begin{align*}
\frac1N\re\rtr{\bf W}^{-1}(z)\bt_{2n}{\bf W}^{-1}(z)\bt_{2n}\bd^{-1}(z)=&\frac1N\rtr{\bf W}^{-1}(z)\bt_{2n}{\bf W}^{-1}(z)\bt_{2n}{\bf W}^{-1}(z)+o(1)\\
=&\int\frac{t^2}{\left(\frac1N\sum_{j=1}^ns_j\psi_j(z)t-z\right)^3}dH_{2n}(t)+o(1).
\end{align*}
and
\begin{align}\label{eq26}
\frac1N\re\rtr{\bf W}^{-1}(z)\bt_{2n}\bd^{-1}(z)=&\frac1N\rtr{\bf W}^{-1}(z)\bt_{2n}{\bf W}^{-1}(z)+o(1)\\
=&\int\frac{t}{\left(\frac1N\sum_{j=1}^ns_j\psi_j(z)t-z\right)^2}dH_{2n}(t)+o(1)\notag.
\end{align}
Together with (\ref{al18}), (\ref{al19})-(\ref{al22}), we have
\begin{align*}
\frac1N\re\rtr\bd^{-1}(z)&\bt_{2n}\bd^{-1}(z)\bt_{2n}=\frac1N\re\rtr{\bd}^{-1}(z)\bt_{2n}{\bf W}^{-1}(z)\bt_{2n}\notag\\
&\times\left[1+\frac1{N^2}\sum_{j=1}^nq_j\left(\re\rtr\bd^{-1}(z)\bt_{2n}\bd^{-1}(z)\bt_{2n}\right)\right]+
o(1)\notag\\
=&\int\frac{t^2}{\left(\frac1N\sum_{j=1}^ns_j\psi_j(z)t-z\right)^2}dH_{2n}(t)\notag\\
&\times\left[1+\frac1{N^2}\sum_{j=1}^nq_j\left(\re\rtr\bd^{-1}(z)\bt_{2n}\bd^{-1}(z)\bt_{2n}\right)\right]
+o(1)
\end{align*}
which yields
\begin{align}\label{eq27}
\frac1N\re\rtr\bd^{-1}(z)&\bt_{2n}\bd^{-1}(z)\bt_{2n}
=\int\frac{t^2}{\left(\frac1N\sum_{j=1}^ns_j\psi_j(z)t-z\right)^2}dH_{2n}(t)\\
&\times\left[1-\frac1{N}\sum_{j=1}^nq_j\int\frac{t^2}
{\left(\frac1N\sum_{j=1}^ns_j\psi_j(z)t-z\right)^2}dH_{2n}(t)\right]^{-1}
+o(1)\notag.
\end{align}
Similarly, we get
\begin{align}\label{eq28}
&\frac1{N}\re\rtr\left(\bw^{-1}(z)\bt_{2n}\bd^{-1}(z)\bt_{2n}\bd^{-1}(z)\right)\\
=&\frac1N\re\rtr{\bf W}^{-1}(z)\bt_{2n}{\bf W}^{-1}(z)\bt_{2n}\bd^{-1}(z)\left[1+\frac1{N^2}\sum_{j=1}^nq_j\left(\re\rtr\bd^{-1}(z)\bt_{2n}\bd^{-1}(z)\bt_{2n}\right)\right]\notag\\
=&\int\frac{t^2}{\left(\frac1N\sum_{j=1}^ns_j\psi_j(z)t-z\right)^3}dH_{2n}(t)
\left[1+\frac1{N^2}\sum_{j=1}^nq_j\left(\re\rtr\bd^{-1}(z)\bt_{2n}\bd^{-1}(z)\bt_{2n}\right)\right]\notag\\
=&\int\frac{t^2}{\left(\frac1N\sum_{j=1}^ns_j\psi_j(z)t-z\right)^3}dH_{2n}(t)\left[1-\frac1{N}\sum_{j=1}^nq_j\int\frac{t^2}
{\left(\frac1N\sum_{j=1}^ns_j\psi_j(z)t-z\right)^2}dH_{2n}(t)\right]^{-1}\notag.
\end{align}
Finally, (\ref{al:3}), (\ref{eq26})-(\ref{eq28}), and
\begin{align*}
\frac1{N}\sum_{j=1}^nq_j=c_n\int\frac{x^2}{\left({1+xg_{2n}^0(z)}\right)^2}dH_{1n}(x)
\end{align*}
one has
\begin{align*}
&d_{n1}(z)=-c_n\int\frac{x^2}{\left({1+xg_{2n}^0(z)}\right)^2}dH_{1n}(x)\int\frac{t^2}{\left(zg_{1n}^0(z)t+z\right)^3}dH_{2n}(t)\\
&\times\left[1-c_n\int\frac{x^2}{\left({1+xg_{2n}^0(z)}\right)^2}dH_{1n}(x)\int\frac{t^2}
{\left(zg_{1n}^0(z)t+z\right)^2}dH_{2n}(t)\right]^{-1}\\
&-c_n\int\frac{x^3}{\left({1+xg_{2n}^0(z)}\right)^3}dH_{1n}(x)\int\frac{t}{\left(zg_{1n}^0(z)t+z\right)^2}dH_{2n}(t)
\int\frac{t^2}{\left(zg_{1n}^0(z)t+z\right)^2}dH_{2n}(t)\\
&\times\left[1-c_n\int\frac{x^2}{\left({1+xg_{2n}^0(z)}\right)^2}dH_{1n}(x)\int\frac{t^2}
{\left(zg_{1n}^0(z)t+z\right)^2}dH_{2n}(t)\right]^{-1}+o(1).
\end{align*}
and
\begin{align*}
&d_{n2}(z)
={- c_n}\int\frac{x^2}{\left(1+xg_{2n}^0(z)\right)^3}dH_{1n}(x)\int\frac{t^2}{\left(zg_{1n}^0(z)t+z\right)^2}dH_{2n}(t)\\
&\times\left[\int\frac{x}{\left({1+xg_{2n}^0(z)}\right)^2}dH_{1n}(x)\right]^{-1}
\int\frac{x^2}{\left({1+xg_{2n}^0(z)}\right)^2}dH_{1n}(x)\int\frac{t}
{\left({z+zg_{1n}^0(z)t}\right)^2}dH_{2n}(t)\\
&\times\left[1-c_n\int\frac{x^2}{\left({1+xg_{2n}^0(z)}\right)^2}dH_{1n}(x)\int\frac{t^2}
{\left(zg_{1n}^0(z)t+z\right)^2}dH_{2n}(t)\right]^{-1}+o(1).
\end{align*}
Consequently, from (\ref{al:10}) and the above two equalities, we conclude that
\begin{align*}
M_{n2}(z)\to&\left(d_1(z)-d_2(z)\right)\Bigg\{1-z^{-1}\left[\int\frac{x}{\left({1+xg_{2}(z)}\right)^2}dH_{1}(x)\right]^{-1}\\
&\times\int\frac{x^2}{\left({1+xg_{2}(z)}\right)^2}dH_{1}(x)\int\frac{t}
{\left({1+g_{1}(z)t}\right)^2}dH_{2}(t)\Bigg\}^{-1}.
\end{align*}

\section{Non-Gaussian case}

It has been verified that Lemma \ref{th2} is true when the entries of the matrix are independent Gaussian variables. This section is to show this conclusion still holds in the general case. The strategy is to compare the characteristic functions of the linear spectral statistics under the normal case and the general case.

 We below assume that $x_{jk},j=1,\cdots,N,k=1,\cdots,n$ are truncated at $\delta_n\sqrt n$, centralized and renormalized as in the last section. That is to say,
\begin{align*}
|x_{jl}|\le\delta_n\sqrt n, \ \re x_{jl}=0, \ \re x_{jl}^2=1, \ \re x_{jl}^4=3+o(1).
\end{align*}

\subsection{From the general case to the Gaussian case}
Denote ${\bf A}_n=\frac1N\bt_{2n}^{1/2}{\bf Y}_n\bt_{1n}{\bf Y}_{n}'\bt_{2n}^{1/2}$ where the entries of ${\bf Y}_n=(y_{jk})$ are independent real Gaussian random variables such that
\begin{align*}
\re y_{jk}=0,\quad \re y_{jk}^2=1,\quad {\rm for} \ j=1\cdots N,k=1,\cdots,n.
\end{align*}
Moreover, suppose that $\bx_n$ and ${\bf Y}_n$ be independent random matrices. As in \cite{distri} for any $\theta\in[0,\pi/2]$, we introduce the following matrices
\begin{align}\label{eq31}
{\bf W}_n(\theta)=\bx_n\sin\theta+{\bf Y}_n\cos\theta\quad{\rm and}\quad{\bf G}_n(\theta)=\frac1N\bt_{2n}^{1/2}{\bf W}_n\bt_{1n}{\bf W}_{n}'\bt_{2n}^{1/2}
\end{align}
where
$$\left({\bf W}_n(\theta)\right)_{jk}=w_{jk}=x_{jk}\sin\theta+{y}_{jk}\cos\theta.$$
Furthermore, let
\begin{align}\label{eq32}
&{\bf H}_n(t,\theta)=e^{it{\bf G}_n(\theta)},\ S(\theta)=\rtr f({\bf G}_n(\theta)),\\
&S^0(\theta)=S(\theta)-N\int f(x)dF^{c_n,H_{1n},H_{2n}}(x), \ Z_{n}(x,\theta)=\re e^{ixS^0(\theta)}\notag.
\end{align}
For simplicity, we omit the argument $\theta$ from the notations of ${\bf W}_n(\theta),{\bf G}_n(\theta),{\bf H}_n(t,\theta)$ and denote them by ${\bf W}_n,{\bf G}_n,{\bf H}_n(t)$ respectively.

Note that
\begin{align}\label{eq29}
Z_n(x,\pi/2)-Z_n(x,0)=\int_0^{\pi/2}\frac{\partial Z_n(x,\theta)}{\partial \theta}d\theta.
\end{align}
The aim is to prove that $\frac{\partial Z_n(x,\theta)}{\partial \theta}$ converges to zero uniformly in $\theta$ over the interval $[0,\pi/2]$,
which ensures Lemma \ref{th2}.

To this end, let $f(\lambda)$ be a smooth function with the Fourier transform given by
\begin{align*}
\widehat f(t)=\frac1{2\pi}\int_{-\infty}^{\infty}f(\lambda)e^{-it\lambda}d\lambda.
\end{align*}
From Lemma \ref{le3}, we have
\begin{align*}
\frac{\partial Z_{n}(x,\theta)}{\partial \theta}=\frac {2x i} N\sum_{j=1}^N\sum_{k=1}^n\re w_{jk}'\left[\bt_{2n}^{1/2}\widetilde f({\bf G}_n)\bt_{2n}^{1/2}{\bf W}_{n}\bt_{1n}\right]_{jk}e^{ixS^0(\theta)}
\end{align*}
where
$$w_{jk}'=\frac{d w_{jk}}{d\theta}=x_{jk}\cos\theta-{y}_{jk}\sin\theta$$
and
\begin{align}\label{eq30}
\widetilde f({\bf G}_n)=i\int_{-\infty}^{\infty}u\widehat f(u){\bf H}_n(u)du.
\end{align}
Let ${\bf W}_{n jk}(x)$ denote the corresponding matrix ${\bf W}_{n}$ with $w_{jk}$ replaced by $x$. And let ${\bf G}_{n jk}(x)=\frac1N\bt_{2n}^{1/2}{\bf W}_{n jk}(x)\bt_{1n}{\bf W}_{n jk}'(x)\bt_{2n}^{1/2}$,
\begin{align*}
\varphi_{jk}(x)=\left[\bt_{2n}^{1/2}\widetilde f({\bf G}_{njk}(x))\bt_{2n}^{1/2}{\bf W}_{njk}(x)\bt_{1n}\right]_{jk}e^{ixS^0({\bf G}_{njk}(x))}.
\end{align*}
By Taylor's formula, one finds
\begin{align*}
\varphi_{jk}(w_{jk})=\sum_{l=0}^3\frac1{l!}w_{jk}^l\varphi_{jk}^{(l)}(0)+\frac1{4!}w_{jk}^4\varphi_{jk}^{(4)}(\varrhoup w_{jk})\quad\varrhoup\in(0,1)
\end{align*}
which implies that
\begin{align*}
\frac{\partial Z_{n}(x,\theta)}{\partial \theta}=\frac {2x i} N\sum_{l=0}^3\frac1{l!}\sum_{j=1}^N\sum_{k=1}^n\re w_{jk}'w_{jk}^l\re\varphi_{jk}^{(l)}(0)+\frac {2x i}{4!N} \sum_{j=1}^N\sum_{k=1}^n w_{jk}' w_{jk}^4\varphi_{jk}^{(4)}(\varrhoup w_{jk}).
\end{align*}
It is easy to obtain
\begin{align*}
&\re w_{jk}'w_{jk}^0=0,\quad\quad\quad\qquad\qquad\re w_{jk}'w_{jk}^1=0,\\
&\re w_{jk}'w_{jk}^2=\re w_{jk}^3\sin^2\theta\cos\theta,\quad\re w_{jk}'w_{jk}^3=o(1)\sin^3\theta\cos\theta.
\end{align*}
It follows that
\begin{align*}
\frac{\partial Z_{n}(x,\theta)}{\partial \theta}=&\frac {x i} {N}\sum_{j=1}^N\sum_{k=1}^n\re w_{jk}^3\sin^2\theta\cos\theta\re\varphi_{jk}^{(2)}(0)
+\frac {x i}{12N} \sum_{j=1}^N\sum_{k=1}^n \re w_{jk}' w_{jk}^4\varphi_{jk}^{(4)}(\varrhoup w_{jk})\\
\triangleq&\mathcal{I}_1+\mathcal{I}_2.
\end{align*}
We below analyze $\mathcal{I}_1$ and $\mathcal{I}_2$ term by term.

\subsubsection{The second derivative}\label{se}
We first consider $\mathcal{I}_1$. A direct calculation yields that
\begin{align*}
\varphi_{jk}^{(2)}(w_{jk})=&\left[\bt_{2n}^{1/2}\frac{\partial ^2\widetilde f({\bf G}_{n})}{\partial w_{jk}^2}\bt_{2n}^{1/2}{\bf W}_{n}\bt_{1n}\right]_{jk}e^{ixS^0(\theta)}+2\left[\bt_{2n}^{1/2}\frac{\partial \widetilde f({\bf G}_{n})}{\partial w_{jk}}\bt_{2n}^{1/2}\right]_{jj}\left[\bt_{1n}\right]_{kk}e^{ixS^0(\theta)}\\
&+\frac{6xi}{N}\left[\bt_{2n}^{1/2}\frac{\partial \widetilde f({\bf G}_{n})}{\partial w_{jk}}\bt_{2n}^{1/2}{\bf W}_n\bt_{1n}\right]_{jk}\left[\bt_{2n}^{1/2} \widetilde f({\bf G}_{n})\bt_{2n}^{1/2}{\bf W}_n\bt_{1n}\right]_{jk}e^{ixS^0(\theta)}\\
&+\frac{6xi}{N}\left[\bt_{2n}^{1/2} \widetilde f({\bf G}_{n})\bt_{2n}^{1/2}{\bf W}_n\bt_{1n}\right]_{jk}\left[\bt_{2n}^{1/2} \widetilde f({\bf G}_{n})\bt_{2n}^{1/2}\right]_{jj}\left[\bt_{1n}\right]_{kk}e^{ixS^0(\theta)}\\
&-\frac{4x^2}{N^2}\left[\bt_{2n}^{1/2} \widetilde f({\bf G}_{n})\bt_{2n}^{1/2}{\bf W}_n\bt_{1n}\right]_{jk}^3e^{ixS^0(\theta)}\\
\triangleq&\mathcal{J}_{jk}^1+\mathcal{J}_{jk}^2+\mathcal{J}_{jk}^3+\mathcal{J}_{jk}^4+\mathcal{J}_{jk}^5.
\end{align*}
 Using Lemma \ref{le3}, one finds
\begin{align*}
\mathcal{J}_{jk}^1
=&-\frac2N\int_{-\infty}^{\infty}u\widehat{f}(u)[\bt_{1n}]_{kk}[\bt_{2n}^{1/2}{\bf H}_n\bt_{2n}^{1/2}]_{jj}*[\bt_{2n}^{1/2}{\bf H}_n\bt_{2n}^{1/2}{\bf W}_n\bt_{1n}]_{jk}(u)e^{ixS_0(\theta)}du\\
&-\frac{6i}{N^2}\int_{-\infty}^{\infty}u\widehat{f}(u)[\bt_{2n}^{1/2}{\bf H}_n\bt_{2n}^{1/2}]_{jj}*[\bt_{2n}^{1/2}{\bf H}_n\bt_{2n}^{1/2}{\bf W}_n\bt_{1n}]_{j k}\\
&*[\bt_{1n}{\bf W}_n'\bt_{2n}^{1/2}{\bf H}_n\bt_{2n}^{1/2}{\bf W}_n\bt_{1n}]_{kk}(u)e^{ixS_0(\theta)}du\\
&-\frac {2i}{N^2}\int_{-\infty}^{\infty}u\widehat{f}(u)[\bt_{2n}^{1/2}{\bf H}_n\bt_{2n}^{1/2}{\bf W}_n\bt_{1n}]_{jk}*[\bt_{2n}^{1/2}{\bf H}_n\bt_{2n}^{1/2}{\bf W}_n\bt_{1n}]_{jk}\\
&*[\bt_{2n}^{1/2}{\bf H}_n\bt_{2n}^{1/2}{\bf W}_n\bt_{1n}]_{jk}(u)e^{ixS_0(\theta)}du.
\end{align*}
It is straightforward to check that the moments of $\left\|\bt_{2n}^{1/2}{\bf H}_n\bt_{2n}^{1/2}\right\|$, $\frac1{\sqrt N}\left\|\bt_{2n}^{1/2}{\bf H}_n\bt_{2n}^{1/2}{\bf W}_n\bt_{1n}\right\|$ and
\begin{align}\label{mm}
\frac1N\left\|\bt_{1n}{\bf W}_n'\bt_{2n}^{1/2}{\bf H}_n\bt_{2n}^{1/2}{\bf W}_n\bt_{1n}\right\|
\end{align}
 are bounded.
Applying Lemma \ref{le4}, we obtain
\begin{align*}
&\left|\frac1{N}\sum_{j=1}^N\sum_{k=1}^n\re\mathcal{J}_{jk}^1\right|
\le\frac C{N^{1/4}}\int_{-\infty}^{\infty}\left(|u|^2+|u|^3\right)|\widehat{f}(u)|du\le C N^{-1/4}.
\end{align*}
By the same argument, we get
\begin{align*}
&\left|\frac1{N}\sum_{j=1}^N\sum_{k=1}^n\re\left(\mathcal{J}_{jk}^2+\mathcal{J}_{jk}^3+\mathcal{J}_{jk}^4+\mathcal{J}_{jk}^5\right)\right|
\le C N^{-1/4}.
\end{align*}
Hence,
\begin{align*}
&\left|\mathcal{I}_1\right|\to0\quad{\rm as}\ n\to\infty.
\end{align*}

\subsubsection{The remainder term}
It is straightforward to check that
\begin{align*}
\re w_{jk}' w_{jk}^4\le C\delta_n\sqrt n.
\end{align*}
Let $w$ be a random variable which has the same first, second and fourth moments as $w_{jk}$. We estimate $\re\sup_{w}\varphi_{jk}^{(4)}(w)$. A direct but tedious computation yields
\begin{align*}
&\varphi_{jk}^{(4)}(w_{jk})\\
=&[\bt_{2n}^{1/2}\frac{\partial ^4\widetilde f({\bf G}_{n})}{\partial w_{jk}^4}\bt_{2n}^{1/2}{\bf W}_n\bt_{1n}]_{jk}e^{ixS_0(\theta)}
+4[\bt_{1n}]_{kk}[\bt_{2n}^{1/2}\frac{\partial ^3\widetilde f({\bf G}_{n})}{\partial w_{jk}^3}\bt_{2n}^{1/2}]_{jj}e^{ixS_0(\theta)}\\
&+\frac {10xi}{N}[\bt_{2n}^{1/2}\frac{\partial ^3\widetilde f({\bf G}_{n})}{\partial w_{jk}^3}\bt_{2n}^{1/2}{\bf W}_n\bt_{1n}]_{jk}[\bt_{2n}^{1/2}{\widetilde f({\bf G}_{n})}\bt_{2n}^{1/2}{\bf W}_n\bt_{1n}]_{jk}e^{ixS_0(\theta)}\\
&+\frac {30xi}{N}[\bt_{1n}]_{kk}[\bt_{2n}^{1/2}\frac{\partial ^2\widetilde f({\bf G}_{n})}{\partial w_{jk}^2}\bt_{2n}^{1/2}]_{jj}[\bt_{2n}^{1/2}{\widetilde f({\bf G}_{n})}\bt_{2n}^{1/2}{\bf W}_n\bt_{1n}]_{jk}e^{ixS_0(\theta)}\\
&+\frac {20xi}{N}[\bt_{2n}^{1/2}\frac{\partial ^2\widetilde f({\bf G}_{n})}{\partial w_{jk}^2}\bt_{2n}^{1/2}{\bf W}_n\bt_{1n}]_{jk}[\bt_{2n}^{1/2}\frac{\partial \widetilde f({\bf G}_{n})}{\partial w_{jk}}\bt_{2n}^{1/2}{\bf W}_n\bt_{1n}]_{jk}e^{ixS_0(\theta)}\\
&+\frac {20xi}{N}[\bt_{1n}]_{kk}[\bt_{2n}^{1/2}{\widetilde f({\bf G}_{n})}\bt_{2n}^{1/2}]_{jj}[\bt_{2n}^{1/2}\frac{\partial ^2\widetilde f({\bf G}_{n})}{\partial w_{jk}^2}\bt_{2n}^{1/2}{\bf W}_n\bt_{1n}]_{jk}e^{ixS_0(\theta)}\\
&+\frac {40xi}{N}[\bt_{1n}]_{kk}[\bt_{2n}^{1/2}\frac{\partial \widetilde f({\bf G}_{n})}{\partial w_{jk}}\bt_{2n}^{1/2}]_{jj}[\bt_{2n}^{1/2}\frac{\partial \widetilde f({\bf G}_{n})}{\partial w_{jk}}\bt_{2n}^{1/2}{\bf W}_n\bt_{1n}]_{jk}e^{ixS_0(\theta)}\\
&+\frac {40xi}{N}[\bt_{1n}]_{kk}^2[\bt_{2n}^{1/2}{\widetilde f({\bf G}_{n})}\bt_{2n}^{1/2}]_{jj}[\bt_{2n}^{1/2}\frac{\partial \widetilde f({\bf G}_{n})}{\partial w_{jk}}\bt_{2n}^{1/2}]_{jj}e^{ixS_0(\theta)}\\
&-\frac {40x^2}{N^2}[\bt_{2n}^{1/2}\frac{\partial ^2\widetilde f({\bf G}_{n})}{\partial w_{jk}^2}\bt_{2n}^{1/2}{\bf W}_n\bt_{1n}]_{jk}[\bt_{2n}^{1/2}{\widetilde f({\bf G}_{n})}\bt_{2n}^{1/2}{\bf W}_n\bt_{1n}]_{jk}^2e^{ixS_0(\theta)}\\
&-\frac {60x^2}{N^2}[\bt_{2n}^{1/2}\frac{\partial \widetilde f({\bf G}_{n})}{\partial w_{jk}}\bt_{2n}^{1/2}{\bf W}_n\bt_{1n}]_{jk}^2[\bt_{2n}^{1/2}{\widetilde f({\bf G}_{n})}\bt_{2n}^{1/2}{\bf W}_n\bt_{1n}]_{jk}e^{ixS_0(\theta)}\\
&-\frac {80x^2}{N^2}[\bt_{1n}]_{kk}[\bt_{2n}^{1/2}\frac{\partial \widetilde f({\bf G}_{n})}{\partial w_{jk}}\bt_{2n}^{1/2}]_{jj}[\bt_{2n}^{1/2}{\widetilde f({\bf G}_{n})}\bt_{2n}^{1/2}{\bf W}_n\bt_{1n}]_{jk}^2e^{ixS_0(\theta)}\\
&-\frac {60x^2}{N^2}[\bt_{1n}]_{kk}^2[\bt_{2n}^{1/2}{ \widetilde f({\bf G}_{n})}\bt_{2n}^{1/2}]_{jj}^2[\bt_{2n}^{1/2}{\widetilde f({\bf G}_{n})}\bt_{2n}^{1/2}{\bf W}_n\bt_{1n}]_{jk}e^{ixS_0(\theta)}\\
&-\frac {80x^3i}{N^3}[\bt_{2n}^{1/2}\frac{\partial \widetilde f({\bf G}_{n})}{\partial w_{jk}}\bt_{2n}^{1/2}{\bf W}_n\bt_{1n}]_{jk}[\bt_{2n}^{1/2}{\widetilde f({\bf G}_{n})}\bt_{2n}^{1/2}{\bf W}_n\bt_{1n}]_{jk}^3e^{ixS_0(\theta)}\\
&-\frac {80x^3i}{N^3}[\bt_{1n}]_{kk}[\bt_{2n}^{1/2}{ \widetilde f({\bf G}_{n})}\bt_{2n}^{1/2}]_{jj}[\bt_{2n}^{1/2}{\widetilde f({\bf G}_{n})}\bt_{2n}^{1/2}{\bf W}_n\bt_{1n}]_{jk}^3e^{ixS_0(\theta)}\\
&-\frac {120x^2}{N^2}[\bt_{1n}]_{kk}[\bt_{2n}^{1/2}\frac{\partial \widetilde f({\bf G}_{n})}{\partial w_{jk}}\bt_{2n}^{1/2}{\bf W}_n\bt_{1n}]_{jk}[\bt_{2n}^{1/2}{\widetilde f({\bf G}_{n})}\bt_{2n}^{1/2}{\bf W}_n\bt_{1n}]_{jk}\\
&\times[\bt_{2n}^{1/2}{ \widetilde f({\bf G}_{n})}\bt_{2n}^{1/2}]_{jj}e^{ixS_0(\theta)}
+\frac{16x^4}{N^4}[\bt_{2n}^{1/2}{\widetilde f({\bf G}_{n})}\bt_{2n}^{1/2}{\bf W}_n\bt_{1n}]_{jk}^5e^{ixS_0(\theta)}.
\end{align*}
We only estimate the first term $[\bt_{2n}^{1/2}\frac{\partial ^4\widetilde f({\bf G}_{n})}{\partial w_{jk}^4}\bt_{2n}^{1/2}{\bf W}_n\bt_{1n}]_{jk}e^{ixS_0(\theta)}$ and the others are similar. Simple calculations imply that
\begin{align*}
&[\bt_{2n}^{1/2}\frac{\partial ^4\widetilde f({\bf G}_{n})}{\partial w_{jk}^4}\bt_{2n}^{1/2}{\bf W}_n\bt_{1n}]_{jk}e^{ixS_0(\theta)}\\
=&-\frac{24i}{N^2}\int_{-\infty}^{\infty}u\widehat{f}(u)[\bt_{1n}]_{kk}^2[\bt_{2n}^{1/2}{\bf H}_n\bt_{2n}^{1/2}]_{jj}*[\bt_{2n}^{1/2}{\bf H}_n\bt_{2n}^{1/2}]_{jj}\\
&*[\bt_{2n}^{1/2}{\bf H}_n\bt_{2n}^{1/2}{\bf W}_n\bt_{1n}]_{jk}(u)e^{ixS_0(\theta)}du\\
&+\frac{144}{N^3}\int_{-\infty}^{\infty}u\widehat{f}(u)[\bt_{1n}]_{kk}[\bt_{2n}^{1/2}{\bf H}_n\bt_{2n}^{1/2}]_{jj}*[\bt_{2n}^{1/2}{\bf H}_n\bt_{2n}^{1/2}{\bf W}_n\bt_{1n}]_{jk}\\
&*[\bt_{2n}^{1/2}{\bf H}_n\bt_{2n}^{1/2}{\bf W}_n\bt_{1n}]_{jk}*[\bt_{2n}^{1/2}{\bf H}_n\bt_{2n}^{1/2}{\bf W}_n\bt_{1n}]_{jk}(u)e^{ixS_0(\theta)}du\\
&+\frac {144}{N^3}\int_{-\infty}^{\infty}u\widehat{f}(u)[\bt_{1n}]_{kk}[\bt_{2n}^{1/2}{\bf H}_n\bt_{2n}^{1/2}]_{jj}*[\bt_{2n}^{1/2}{\bf H}_n\bt_{2n}^{1/2}]_{jj}\\
&*[\bt_{2n}^{1/2}{\bf H}_n\bt_{2n}^{1/2}{\bf W}_n\bt_{1n}]_{jk}
*[\bt_{1n}{\bf W}_n'\bt_{2n}^{1/2}{\bf H}_n\bt_{2n}^{1/2}{\bf W}_n\bt_{1n}]_{kk}(u)e^{ixS_0(\theta)}du\\
&+\frac {240i}{N^4}\int_{-\infty}^{\infty}u\widehat{f}(u)[\bt_{2n}^{1/2}{\bf H}_n\bt_{2n}^{1/2}]_{jj}*[\bt_{2n}^{1/2}{\bf H}_n\bt_{2n}^{1/2}{\bf W}_n\bt_{1n}]_{jk}*[\bt_{2n}^{1/2}{\bf H}_n\bt_{2n}^{1/2}{\bf W}_n\bt_{1n}]_{jk}\\
&*[\bt_{2n}^{1/2}{\bf H}_n\bt_{2n}^{1/2}{\bf W}_n\bt_{1n}]_{jk}
*[\bt_{1n}{\bf W}_n'\bt_{2n}^{1/2}{\bf H}_n\bt_{2n}^{1/2}{\bf W}_n\bt_{1n}]_{kk}(u)e^{ixS_0(\theta)}du\\
&+\frac {120i}{N^4}\int_{-\infty}^{\infty}u\widehat{f}(u)[\bt_{2n}^{1/2}{\bf H}_n\bt_{2n}^{1/2}]_{jj}*[\bt_{2n}^{1/2}{\bf H}_n\bt_{2n}^{1/2}]_{jj}*[\bt_{2n}^{1/2}{\bf H}_n\bt_{2n}^{1/2}{\bf W}_n\bt_{1n}]_{jk}\\
&*[\bt_{1n}{\bf W}_n'\bt_{2n}^{1/2}{\bf H}_n\bt_{2n}^{1/2}{\bf W}_n\bt_{1n}]_{kk}
*[\bt_{1n}{\bf W}_n'\bt_{2n}^{1/2}{\bf H}_n\bt_{2n}^{1/2}{\bf W}_n\bt_{1n}]_{kk}(u)e^{ixS_0(\theta)}du\\
&+\frac {24i}{N^4}\int_{-\infty}^{\infty}u\widehat{f}(u)[\bt_{2n}^{1/2}{\bf H}_n\bt_{2n}^{1/2}{\bf W}_n\bt_{1n}]_{jk}*[\bt_{2n}^{1/2}{\bf H}_n\bt_{2n}^{1/2}{\bf W}_n\bt_{1n}]_{jk}*[\bt_{2n}^{1/2}{\bf H}_n\bt_{2n}^{1/2}{\bf W}_n\bt_{1n}]_{jk}\\
&*[\bt_{2n}^{1/2}{\bf H}_n\bt_{2n}^{1/2}{\bf W}_n\bt_{1n}]_{jk}*[\bt_{2n}^{1/2}{\bf H}_n\bt_{2n}^{1/2}{\bf W}_n\bt_{1n}]_{jk}(u)e^{ixS_0(\theta)}du\\
\triangleq&\mathcal{Q}_{jk}^1+\mathcal{Q}_{jk}^2+\mathcal{Q}_{jk}^3+\mathcal{Q}_{jk}^4+\mathcal{Q}_{jk}^5+\mathcal{Q}_{jk}^6.
\end{align*}
From (\ref{mm}), we deduce that
\begin{align*}
\left|\frac1N\sum_{j,k}\re \mathcal{Q}_{jk}^1\right|\le& \frac C{ N}\int_{-\infty}^{\infty}|u|^3\widehat{f}(u)\re\left\|\bt_{2n}^{1/2}{\bf H}_n\bt_{2n}^{1/2}\right\|^2\left\|\bt_{2n}^{1/2}{\bf H}_n\bt_{2n}^{1/2}{\bf W}_n\bt_{1n}\right\|du\\
\le& \frac C{\sqrt N}\int_{-\infty}^{\infty}|u|^3\widehat{f}(u)du=O(N^{-1/2}).
\end{align*}
Applying the same arguments as above one can conclude that
\begin{align*}
&\left|\frac1N\sum_{j,k}\re \left(\mathcal{Q}_{jk}^2+\mathcal{Q}_{jk}^3+\mathcal{Q}_{jk}^4+\mathcal{Q}_{jk}^5+\mathcal{Q}_{jk}^6\right)\right|\\
\le& \frac C{\sqrt N}\int_{-\infty}^{\infty}\left(|u|^4+|u|^5\right)\widehat{f}(u)du=O(N^{-1/2}).
\end{align*}
It follows that
\begin{align*}
|\mathcal{I}_2|\le C\delta_n\to0.
\end{align*}
This fact finishes the proof of Lemma \ref{th2}.
\begin{appendix}

\section{  }
This section is to prove some lemmas which are used in the proof of Lemma \ref{th2}.

\begin{lemma}\label{add}
Under the conditions of Theorem \ref{th1}, we have for $z\in\mathcal{C}_u$ and $p\ge1$
\begin{align*}
\re|\beta_k(z)|^p\le C, \ \re|\widetilde\beta_k(z)|\le C, \ |b_k(z)|\le C, \ |\psi_k(z)|\le C
\end{align*}
where $\beta_k(z),\widetilde\beta_k(z),b_k(z),\psi_k(z)$ are defined in (\ref{eq5}) and (\ref{eq6}).
\end{lemma}
\begin{proof}
We only prove $\re|\beta_k(z)|^p\le C$ and the others are similar. Note that
\begin{align*}
N^{-1}\left|\by_k'\bd_k^{-1}(z)\by_k\right|\le\frac{\|\by_k\|^2}{Nv_0}
\end{align*}
which gives
\begin{align}\label{eq7}
|\beta_k(z)|\le\frac1{1-\frac{\left|s_k\right|\|\by_k\|^2}{Nv_0}}<2\quad{\rm if}\ \frac{\left|s_k\right|\|\by_k\|^2}{Nv_0}\le1/2
\end{align}
where $\|\by_k\|^2=\sum_{j=1}^N y_{j k}^2$. Denoting by ${\bf O}\boldsymbol\Lambda{\bf O}^*$ the spectral decomposition of $\bb_{nk}$ and $\boldsymbol\Lambda={\rm diag}\left(\lambda_1,\cdots,\lambda_N\right)$, we obtain
\begin{align}\label{eq8}
|\beta_k(z)|\le&\frac1{\left|\Im\left(N^{-1}s_k\by_k'\bd_k^{-1}(z)\by_k\right)\right|}
=\frac1{\left|N^{-1}s_kv_0\sum_{j=1}^N\frac{\left({\bf O}^*\by_k\by_k'{\bf O}\right)_{jj}}{\left(\lambda_j-u\right)^2+v_0^2}\right|}\notag\\
\le&\frac{{2\max_{1\le j\le N}\lambda_j^2+2|z|^2}}{N^{-1}v_0\|\by_k\|^2\left|s_k\right|}
\le\frac{{4\max_{1\le j\le N}\lambda_j^2+4|z|^2}}{v_0^2}\quad{\rm if}\ \frac{\left|s_k\right|\|\by_k\|^2}{Nv_0}>1/2.
\end{align}
Combining (\ref{eq7}) with (\ref{eq8}), we have
\begin{align*}
|\beta_k(z)|
\le\frac{{4\max_{1\le j\le N}\lambda_j^2+4|z|^2}}{v_0^2}+2\le\frac{{4\max_{1\le j\le N}\lambda_j^2+6|z|^2}}{v_0^2}.
\end{align*}
Moreover,
\begin{align*}
\max_{1\le j\le N}\lambda_j^2\le \rtr \bb_{nk}^2\le \frac{\tau^4}{N}\left({\bf x}_k'{\bf x}_k\right)^2\le \tau^4\delta_n^4Nn^2\le CN^3.
\end{align*}
It follows that
\begin{align}\label{eq9}
|\beta_k(z)|\le\begin{cases}
\frac{{4\eta_l^2+4\eta_r^2+6|z|^2}}{v_0^2}\le C,&{\rm if} \ \eta_l\le \lambda_{\min}\le\lambda_{\max}\le\eta_r,\\
CN^3,&{\rm otherwise}.
\end{cases}
\end{align}
Using (\ref{max}), (\ref{min}) and (\ref{eq9}), we have for suitably large $l$
\begin{align*}
\re|\beta_k(z)|^p\le&\re|\beta_k(z)|^pI(\eta_l\le \lambda_{\min}\le\lambda_{\max}\le\eta_r)+\re|\beta_k(z)|^pI(\lambda_{\min}<\eta_l \ {\rm or} \ \lambda_{\max}>\eta_r)\\
\le&C+{{CN^3}}{\rm P}(\lambda_{\min}<\eta_l \ {\rm or} \ \lambda_{\max}>\eta_r)\\
\le&C+{{CN^3}}n^{-l}\le C.
\end{align*}
This completes the proof of this lemma.
\end{proof}

\begin{lemma}\label{lep10}
Under the conditions of Theorem \ref{th1}, we have
\begin{align*}
\sup_{z\in\mathcal{C}_n}|\re g_{2n}(z)-g_{2n}^0(z)|\to0.
\end{align*}
\end{lemma}
\begin{proof}
 Similarly to the proof of (4.1) in \cite{bai2004clt}, one can check that
\begin{align*}
\sup_{z\in\mathcal{C}_n}|\re \underline m_n(z)-\underline m(z)|\to0\quad{\rm as}\ n\to\infty
\end{align*}
By Lemma \ref{lep1}, (\ref{al:5}) and (\ref{al:8}), we have
\begin{align*}
\sup_{z\in\mathcal{C}_n}\left|\frac1{n}\sum_{j=1}^n\left(\re\beta_j(z)-\psi_j(z)\right)\right|\le C |\re\beta_j(z)-\psi_j(z)|
\le \frac CN\to0.
\end{align*}
Note that
\begin{align*}
&|\re g_{2n}(z)-g_{2n}^0(z)|\left|\frac1{n}\sum_{j=1}^n\frac{s_j}{\left({1+s_j\re g_{2n}(z)}\right)\left({1+s_jg_{2n}^0(z)}\right)}\right|\\
=&
\left|\frac1{n}\sum_{j=1}^n\left(\psi_j(z)-\frac1{1+s_jg_{2n}^0(z)}\right)\right|\\
\le&\left|\frac1{n}\sum_{j=1}^n\left(\re\beta_j(z)-\psi_j(z)\right)\right|
+|z||\re \underline m_n(z)-\underline m(z)|+|z||\underline m_n^0(z)-\underline m(z) |\to 0.
\end{align*}
Consequently, it suffices to prove that $\left|\frac1{n}\sum_{j=1}^n\frac{s_j}{\left({1+s_j\re g_{2n}^0(z)}\right)\left({1+s_jg_{2n}^0(z)}\right)}\right|$ has a lower bound. We prove it by contradiction. By (\ref{i1}), one has
\begin{align}\label{ab}
(z\underline m (z))'=g_2'(z)\int\frac{x}{(1+g_2(z)x)^2}dH_1.
\end{align}
Suppose that there exists a sequence $\{z_h\in\mathcal{C}_n\}$ such that $z_h\to z_{0}$ and
\begin{align*}
\int\frac{x}{(1+g_2(z_h)x)^2}dH_1(x)\to0.
\end{align*}
From (\ref{ab}) and the continuity of $g_2(z)$, it follows that
\begin{align}\label{ac}
(z_h\underline m (z_h))'\to (z_{0}\underline m (z_{0}))'=0.
\end{align}
Then from the equation $m(z_0)+z_0m'(z_0)=0$, we obtain another solution of $m(z)$ which has nothing to do with $H_1(z)$ and $H_2(z)$. However,
this contradicts to the fact that $m(z)$ is a unique solution of (\ref{eqs}). Hence, we get
\begin{align*}
\left|\int\frac{x}{(1+g_2(z)x)^2}dH_1(x)\right|>0.
\end{align*}
By continuity and convergence of $\re g_{2n}(z)$, we see that for all large $n$
$$\left|\frac1{n}\sum_{j=1}^n\frac{s_j}{\left({1+s_j\re g_{2n}(z)}\right)\left({1+s_jg_{2n}^0(z)}\right)}\right|>0.$$
Therefore, we conclude that
\begin{align*}
\sup_{z\in\mathcal{C}_n}|\re g_{2n}(z)-g_{2n}^0(z)|\to0.
\end{align*}
\end{proof}

\begin{lemma}\label{lep7}
For $z\in\mathcal{C}_n$, we have for any positive $p\ge1$
\begin{align*}
\re|\beta_j(z)|^p \le C^p
\end{align*}
where $\beta_j(z)$ is defined in (\ref{eq5}).
\end{lemma}
\begin{proof}
By formula (\ref{al2}), we get
\begin{align*}
\bd^{-1}(z)-\bd_j^{-1}(z)=-\frac1Ns_j\beta_j(z)\bd^{-1}_j\by_j\by_j'\bd_j^{-1}.
\end{align*}
This yields
\begin{align*}
\frac1Ns_j\by_j'\bd^{-1}(z)\by_j=&\frac1Ns_j\by_j'\bd_j^{-1}(z)\by_j\left(1-\frac1Ns_j\beta_j(z)\by_j'\bd^{-1}_j\by_j\right)\\
=&\frac1Ns_j\beta_j(z)\by_j'\bd_j^{-1}(z)\by_j=1-\beta_j(z).
\end{align*}
If $\eta_l<\lambda_{\min}^{\bb_n}\le\lambda_{\max}^{\bb_n}<\eta_r$ and $\eta_l<\lambda_{\min}^{\bb_{nj}}\le\lambda_{\max}^{\bb_{nj}}<\eta_r$, then we have
\begin{align*}
|\frac1Ns_j\by_j'\by_j|=&\left|\max_{\|{\bf f}\|=1}{\bf f}'\left(\bb_n-\bb_{nj}\right){\bf f}\right|\le
\left|\max_{\|{\bf f}\|=1}{\bf f}'\bb_n{\bf f}\right|+\left|\min_{\|{\bf f}\|=1}{\bf f}'\bb_{nj}{\bf f}\right|\\
=&\left|\lambda_{\max}^{\bb_n}\right|+\left|\lambda_{\min}^{\bb_{nj}}\right|\le2(|\eta_r|+|\eta_l|).
\end{align*}
Otherwise,
\begin{align*}
|\frac1Ns_j\by_j'\by_j|\le\frac{\tau^2}{N}\sum_{k=1}^N|x_{jk}|^2\le\tau^2\delta_n^2n\le n.
\end{align*}
Therefore, one has
\begin{align*}
|\beta_j(z)|=&|1-\frac1Ns_j\by_j'\bd^{-1}(z)\by_j|\le1+2(|\eta_r|+|\eta_l|)\max\{\frac1{x_r-\eta_r},\frac1{\eta_l-x_l},\frac1{v_0}\}\\
&+n^{2+\alpha}I(\lambda_{\min}^{\bb_n}<\eta_l \ {\rm or} \ \lambda_{\max}^{\bb_n}<\eta_r \ {\rm or} \ \lambda_{\min}^{\bb_{nj}}>\eta_l \ {\rm or} \ \lambda_{\max}^{\bb_{nj}}<\eta_r).
\end{align*}
By (\ref{max}) and (\ref{min}), we have for any positive $p\ge1$ and $l>3$
\begin{align*}
\re|\beta_j(z)|^p \le C_1^p+C_2^pn^{-l}\le C^p.
\end{align*}
\end{proof}

\begin{lemma}\label{le1}
Let $1\le j\le N,1\le k\le n$. Recall the definition of ${\bf G}_n$ in (\ref{eq31}). Then, for any $1\le a,b \le N$, we have
\begin{align*}
\frac{\partial g_{ab}}{\partial w_{jk}}=\left[\frac{\partial {\bf G}_n}{\partial w_{jk}}\right]_{ab}=\frac1N\left[\bt_{2n}^{1/2}\right]_{aj}\left[\bt_{1n}{\bf W}_{n}'\bt_{2n}^{1/2}\right]_{kb}
+\frac1N\left[\bt_{2n}^{1/2}{\bf W}_{n}\bt_{1n}\right]_{ak}\left[\bt_{2n}^{1/2}\right]_{jb}.
\end{align*}
\end{lemma}
\begin{proof}
It is obvious that
\begin{align*}
\frac{\partial {\bf G}_n}{\partial w_{jk}}=&\frac1N\bt_{2n}^{1/2}\frac{\partial{\bf W}_n}{\partial w_{jk}}\bt_{1n}{\bf W}_{n}'\bt_{2n}^{1/2}
+\frac1N\bt_{2n}^{1/2}{\bf W}_{n}\bt_{1n}\frac{\partial{\bf W}_n'}{\partial w_{jk}}\bt_{2n}^{1/2}\\
=&\frac1N\bt_{2n}^{1/2}{\bf e}_j{\bf e}_k'\bt_{1n}{\bf W}_{n}'\bt_{2n}^{1/2}
+\frac1N\bt_{2n}^{1/2}{\bf W}_{n}\bt_{1n}{\bf e}_k{\bf e}_j'\bt_{2n}^{1/2}.
\end{align*}
This yields
\begin{align*}
\left[\frac{\partial {\bf G}_n}{\partial w_{jk}}\right]_{ab}=&\frac1N\left[\bt_{2n}^{1/2}\right]_{aj}\left[\bt_{1n}{\bf W}_{n}'\bt_{2n}^{1/2}\right]_{kb}
+\frac1N\left[\bt_{2n}^{1/2}{\bf W}_{n}\bt_{1n}\right]_{ak}\left[\bt_{2n}^{1/2}\right]_{jb}.
\end{align*}
\end{proof}

\begin{lemma}\label{le2}
Let $1\le j\le N,1\le k\le n$. Recall the definition of ${\bf H}_n(t)$ in (\ref{eq32}). Then for any $1\le d,l \le N$
\begin{align*}
\frac{\partial h_{dl}}{\partial w_{jk}}=\frac i N\left[{\bf H}_n\bt_{2n}^{1/2}\right]_{dj}*\left[\bt_{1n}{\bf W}_{n}'\bt_{2n}^{1/2}{\bf H}_n\right]_{kl}(t)+\frac i N\left[{\bf H}_n\bt_{2n}^{1/2}{\bf W}_{n}\bt_{1n}\right]_{dk}*\left[\bt_{2n}^{1/2}{\bf H}_n\right]_{jl}(t)
\end{align*}
where $h_{dl}=\left({\bf H}_n(t)\right)_{dl}$ and $f*g(t)=\int_0^tf(s)g(t-s)ds$.
\end{lemma}
\begin{proof}
Applying Lemma \ref{le1} and Lemma \ref{lea1}, we get
\begin{align*}
\frac{\partial {\bf H}_n(t)}{\partial w_{jk}}=&\sum_{a,b=1}^N\frac{\partial {\bf H}_n(t)}{\partial g_{ab}}\frac{\partial g_{ab}}{\partial w_{jk}}
=\frac i N\sum_{a,b=1}^N\int_0^te^{is{\bf G}_n}{\bf e}_j{\bf e}_k'e^{(1-s){\bf G}_n}ds\\
&\times\left\{\left[\bt_{2n}^{1/2}\right]_{aj}\left[\bt_{1n}{\bf W}_{n}'\bt_{2n}^{1/2}\right]_{kb}
+\left[\bt_{2n}^{1/2}{\bf W}_{n}\bt_{1n}\right]_{ak}\left[\bt_{2n}^{1/2}\right]_{jb}\right\}.
\end{align*}
Hence, one has
\begin{align*}
\frac{\partial h_{dl}}{\partial w_{jk}}
=&\frac i N\sum_{a,b=1}^Nh_{da}*h_{bl}(t)
\left\{\left[\bt_{2n}^{1/2}\right]_{aj}\left[\bt_{1n}{\bf W}_{n}'\bt_{2n}^{1/2}\right]_{kb}
+\left[\bt_{2n}^{1/2}{\bf W}_{n}\bt_{1n}\right]_{ak}\left[\bt_{2n}^{1/2}\right]_{jb}\right\}\\
=&\frac i N\left[{\bf H}_n\bt_{2n}^{1/2}\right]_{dj}*\left[\bt_{1n}{\bf W}_{n}'\bt_{2n}^{1/2}{\bf H}_n\right]_{kl}(t)+\frac i N\left[{\bf H}_n\bt_{2n}^{1/2}{\bf W}_{n}\bt_{1n}\right]_{dk}*\left[\bt_{2n}^{1/2}{\bf H}_n\right]_{jl}(t).
\end{align*}
\end{proof}

\begin{lemma}\label{le3}
Let $1\le j\le N,1\le k\le n$. Recall the definitions of $S(\theta)$ in (\ref{eq32}) and $\widetilde f({\bf G}_n)$ in (\ref{eq30}). Then for any $1\le d,l \le N$
\begin{align*}
\frac{\partial S(\theta)}{\partial w_{jk}}=\frac {2} N\left[\bt_{2n}^{1/2}\widetilde f({\bf G}_n)\bt_{2n}^{1/2}{\bf W}_{n}\bt_{1n}\right]_{jk}.
\end{align*}

\end{lemma}
\begin{proof}
By the inverse Fourier transform, we obtain
\begin{align*}
\frac{\partial S}{\partial w_{jk}}=\int_{-\infty}^{\infty}\widehat f(u)\rtr\frac{\partial {\bf H}_n(u)}{\partial w_{jk}}du.
\end{align*}
It follows from Lemma \ref{le2} that
\begin{align*}
\frac{\partial S}{\partial w_{jk}}=&\frac {2i} N\int_{-\infty}^{\infty}\widehat f(u)\sum_{d=1}^N\left[{\bf H}_n\bt_{2n}^{1/2}\right]_{dj}*\left[\bt_{1n}{\bf W}_{n}'\bt_{2n}^{1/2}{\bf H}_n\right]_{kd}(u)du\\
=&\frac {2i} N\int_{-\infty}^{\infty}u\widehat f(u)\left[\bt_{2n}^{1/2}{\bf H}_n(u)\bt_{2n}^{1/2}{\bf W}_{n}\bt_{1n}\right]_{jk}du\\
=&\frac {2} N\left[\bt_{2n}^{1/2}\widetilde f({\bf G}_n)\bt_{2n}^{1/2}{\bf W}_{n}\bt_{1n}\right]_{jk}.
\end{align*}
\end{proof}

\begin{lemma}\label{le4}
Suppose ${\bf A},{\bf B},{\bf C},{\bf D}$ are $p\times n$ random matrices and their moments of the spectral norms are bounded. Then we get
\begin{align}
&\left|\sum_{j,k}\re{\bf A}_{jj}{\bf B}_{jk}{\bf C}_{kk}\right|\le C n^{5/4}\\
&\left|\sum_{j,k}\re{\bf A}_{jk}{\bf B}_{jk}{\bf C}_{jk}\right|\le C n^{5/4}.
\end{align}
\end{lemma}
\begin{proof}
Using the Cauchy-Schwarz inequality, we have
\begin{align*}
\left|\sum_{j,k}\re{\bf A}_{jj}{\bf B}_{jk}{\bf C}_{kk}\right|\le &
\left(\sum_{j}\re\left|{\bf A}_{jj}\right|^2\right)^{1/2}\left(\sum_{j}\re\left|\sum_{k}{\bf B}_{jk}{\bf C}_{kk}\right|^2\right)^{1/2}\\
\le&C\sqrt n\left(\sum_{k_1,k_2}\re\left({\bf B}'{\bf B}\right)_{k_1k_2}{\bf C}_{k_1k_1}{\bf C}'_{k_2k_2}\right)^{1/2}\\
\le&C\sqrt n\left(\sum_{k_1,k_2}\re\left({\bf B}'{\bf B}\right)^2_{k_1k_2}\right)^{1/4}\left(\sum_{k_1,k_2}\re\left|{\bf C}_{k_1k_1}\right|^2\left|{\bf C}'_{k_2k_2}\right|^2\right)^{1/4}\\
\le&C n^{5/4}
\end{align*}
and
\begin{align*}
\left|\sum_{j,k}\re{\bf A}_{jk}{\bf B}_{jk}{\bf C}_{jk}\right|\le &\left(\sum_{j,k}\re\left|{\bf A}_{jk}\right|\left|{\bf B}_{jk}\right|\right)^{1/2}\left(\sum_{j,k}\re\left|{\bf A}_{jk}\right|\left|{\bf B}_{jk}\right|\left|{\bf C}_{jk}\right|^2\right)^{1/2}\\
\le &\left(\sum_{j,k}\re\left|{\bf A}_{jk}\right|\left|{\bf B}_{jk}\right|\right)^{3/4}\left(\sum_{j,k}\re\left|{\bf A}_{jk}\right|\left|{\bf B}_{jk}\right|\left|{\bf C}_{jk}\right|^4\right)^{1/4}\\
\le &C\sqrt n\left(\sum_{j,k}\re\left|{\bf A}_{jk}\right|^2\right)^{3/8}\left(\sum_{j,k}\re\left|{\bf B}_{jk}\right|^2\right)^{3/8}\\
\le&C n^{5/4}.
\end{align*}
\end{proof}

\section{}
In this section, we list several technical facts that will be often used in the paper.
\begin{lemma}[Lemma B.26 in \cite{bai2004clt}]\label{lep1}
Let ${\bf A}=(a_{jk})$ be an $n\times n$ nonrandom matrix and $\bx=(x_1,\cdots,x_n)'$ be a random vector of independent entries. Assume that $\re x_j=0$, $\re|x_j|^2=1$ and $\re|x_j|^l\le \nu_l$. Then for $p\ge1$,
\begin{align*}
&\re\left|\bx'{\bf A}\bx-\rtr{\bf A}\right|^p
\le C_p\left[\left(\nu_4\rtr{\bf A}{\bf A}'\right)^{p/2}+\nu_{2p}\rtr\left({\bf A}{\bf A}'\right)^{p/2}\right]
\end{align*}
where $C_p$ is a constant depending on $p$ only.
\end{lemma}

\begin{lemma}[Lemma 2.4 in \cite{bai2004clt}]\label{lep2}
Suppose for each $n$ $Y_{n1},Y_{n2},\cdots,Y_{nr_n}$ is a real martingale difference sequence with respect to the increasing $\sigma$-field $\{\mathcal{F}_{nj}\}$ having second moments. If as $n\to\infty$,
\begin{align*}
&(i)\quad\quad\quad\quad\qquad\qquad\sum_{j=1}^{r_n}\re(Y_{nj}^2|\mathcal{F}_{n,j-1})\xrightarrow{i.p.}\sigma^2,\qquad\qquad\qquad\qquad\qquad\qquad\\
&{\rm where}\ \sigma^2\ {\rm is\ a\ positive\ constant}, \ {\rm and\ for\ each}\ \varepsilon\ge0,\\
&(ii)\quad\quad\quad\quad\qquad\qquad\sum_{j=1}^{r_n}\re(Y_{nj}^2I(|Y_{nj}\ge\varepsilon|))\to0,\qquad\qquad\qquad\qquad\qquad\qquad\\
&{\rm then}\\
&\quad\quad\quad\quad\qquad\qquad\qquad\sum_{j=1}^{r_n}Y_{nj}\xrightarrow{D}N(0,\sigma^2).\qquad\qquad\qquad\qquad
\end{align*}
\end{lemma}

\begin{lemma}[Lemma 2.6 in \cite{silverstein1995empirical}]\label{lep3}
Let $z\in\mathbb{C}^+$ with $v=\Im z$, ${\bf A}$ and ${\bf B}$ $N\times N$ with ${\bf B}$ Hermitian, $\tau\in\mathbb{R}$, and ${\rm q}\in\mathbb{C}^N$. Then
\begin{align*}
\left|\rtr\left[\left((\bb-z\bi)^{-1}-(\bb+\tau{\bf qq}^*-z\bi)^{-1}\right){\bf A}\right]\right|\le\frac{\|{\bf A}\|}v.
\end{align*}
\end{lemma}

\begin{lemma}[Abel lemma]\label{lep4}
Suppose $\{f_k\}$ and $\{r_k\}$ are two sequences. Then, we have
\begin{align*}
\sum_{k=1}^nf_k(r_{k+1}-r_k)=f_{n+1}r_{n+1}-f_1r_1-\sum_{k=1}^nr_{k+1}(f_{k+1}-f_k).
\end{align*}
\end{lemma}

\begin{lemma}[inequality (4.8) in \cite{bai2004clt}]\label{lep6}
Let ${\bf M}$ be $N\times N$ nonrandom matrix, we find for $j\in\left\{1,2,\cdots,n\right\}$
\begin{align*}
\re\left|\rtr\bd_j^{-1}{\bf M}-\re\rtr\bd_j^{-1}{\bf M}\right|^2\le C\|{\bf M}\|^2
\end{align*}
\end{lemma}

\begin{lemma}[Theorem A.37 in \cite{bai2010spectral}]\label{lep8}
If $\ba$ and $\bb$ are two $ n\times p$ matrices and $\lambda_k,\delta_k,k=1,2,\cdots,n$ denote their singular values. If the singular values are arranged in  descending order, then we have
\begin{align*}
\sum_{k=1}^\nu|\lambda_k-\delta_k|^2\le\rtr\left[\left(\ba-\bb\right)\left(\ba-\bb\right)^*\right]
\end{align*}
where $\nu=\min\{p,n\}$.
\end{lemma}

\begin{lemma}\label{lep9}
For rectangular matrices $\ba,\bb,{\bf C},{\bf D}$, we have
\begin{align*}
\left|\rtr\left({\bf ABCD}\right)\right|\le \|\ba\|\|{\bf C}\|\left(\rtr{\bf BB}^*\right)^{1/2}\left(\rtr{\bf DD}^*\right)^{1/2}.
\end{align*}
\end{lemma}

\begin{lemma}[Duhamel formula]\label{lea1}
Let ${\bf M}_1,{\bf M}_2$ be $n\times n$ matrices and $t\in\mathbb{R}$. Then we have
\begin{align*}
e^{({\bf M}_1+{\bf M}_2)t}=e^{{\bf M}_1t}+\int_0^te^{{\bf M}_1(t-s)}{\bf M}_2e^{({\bf M}_1+{\bf M}_2)s}ds.
\end{align*}
Moreover, if ${\bf A}(t)$ is a matrix-valued function of $t\in\mathbb{R}$ that is $C^{\infty}$ in the sense that each matrix element $\left[{\bf A}(t)\right]_{jk}$ is $C^{\infty}$. Then
\begin{align*}
\frac{de^{{\bf A}(t)}}{dt}=\int_0^1e^{s{\bf A}(t)}{\bf A}'(t)e^{(1-s){\bf A}(t)}ds.
\end{align*}
\end{lemma}

\end{appendix}


\end{document}